\documentclass[reqno,11pt]{amsart}
\usepackage[dvipsnames,usenames]{color} 
\usepackage[toc,page]{appendix}

\usepackage{enumitem,stmaryrd}
 \usepackage{amsaddr}

\usepackage{hyperref}
\usepackage{graphicx}  
\usepackage[all]{xy} \xyoption{arc} \xyoption{color}

\usepackage{amsmath} 
\usepackage{amsthm} 
\usepackage{amsfonts}  
\usepackage{amssymb} 

\usepackage{verbatim}

\newcommand{\set}[1]{ \{#1\} } 
   
\newcommand{\bset}[1]{ [#1] }

\newcommand{\pset}[1]{ (#1) }

\newcommand{\Pset}[1]{ \left(#1\right) } 
\newcommand{\norm}[1]{ \|#1\| }

\newcommand{\abs}[1]{ |#1| } 
\newcommand{\Abs}[1]{\left |#1\right| }

\newcommand{\cp}[1]{,_{#1}}

\def\hd{\bar{\partial}} 
\def\Div{ \operatorname{div} }

\def\local{l} 

\def\UL{U_\local}

\def\thetal{\theta_\local}

\numberwithin{equation}{section}

\newtheorem{theorem}{\sc Theorem}[section]
\newtheorem{lemma}{\sc Lemma}[section]
\newtheorem{proposition}{\sc Proposition}[section]
\theoremstyle{definition}
\newtheorem{remark}{\sc Remark}
\theoremstyle{definition}
\newtheorem{definition}{\sc Definition}[section]

\def\R{\mathbb R}

\def\n{\nonumber}

\def\p{\tilde p}

\def\R{\mathbb R}

\def\n{\nonumber}

\def\curl{\operatorname{curl}}

\def\R{\mathbb R}
\def\n{\nonumber}

\def\p{\tilde p}

\def\p{\partial}

\def\nabla{D}

\title[Splash and splat singularities for the 3-D Euler equations]
{On the finite-time splash and splat singularities for the 3-D free-surface Euler equations}

\author{Daniel Coutand and Steve Shkoller}

\address{\large Accepted for publication in Communications in Mathematical Physics}


\email{D.Coutand@ma.hw.ac.uk}
\email{shkoller@math.ucdavis.edu}

\begin{document}

\begin{abstract}
We prove that the 3-D free-surface incompressible Euler equations with regular initial geometries and velocity fields  have solutions which can form a finite-time ``splash'' (or ``splat'') singularity first introduced in \cite{CaCoFeGaGo2011}, wherein the evolving 2-D hypersurface, the moving boundary of the fluid domain,  self-intersects at a point (or on surface).   Such singularities can occur when the crest of a breaking wave  falls unto its trough, or in the study of drop impact upon liquid surfaces.    Our approach is founded upon
the Lagrangian description of the free-boundary problem, combined with a novel approximation scheme of a finite collection
of local coordinate charts; as such we are able to analyze a rather general set of geometries for the evolving 2-D free-surface of
the fluid.   We do not assume the fluid is irrotational, and as such, our method can be used for a number of other fluid interface problems, including compressible flows, plasmas, as well as the inclusion of surface tension effects.
\end{abstract}

\maketitle

\tableofcontents

\section{Introduction}
\label{sec_introduction}

\subsection{The Eulerian description of the free-boundary problem}
For $0 \le t \le T$,
the evolution of a three-di\-men\-si\-o\-nal  incompressible fluid
with a moving free-surface  is modeled by the in\-com\-pres\-sib\-le Euler equations:
\begin{subequations}
  \label{euler}
\begin{alignat}{2}
u_t+ u\cdot Du + D p&=0  &&\text{in} \ \ \Omega(t) \,, \label{euler.a}\\
  {\operatorname{div}} u &=0
&&\text{in} \ \ \Omega(t) \,, \label{euler.b}\\
p &= 0 \ \ &&\text{on} \ \ \Gamma(t) \,, \label{euler.c}\\
\mathcal{V} (\Gamma(t))& = u \cdot n &&\ \ \label{euler.d}\\
u   &= u_0  \ \  &&\text{on} \ \ \Omega(0) \,, \label{euler.e}\\
   \Omega(0) &= \Omega_0\,.  && \label{ceuler.f}
\end{alignat}
\end{subequations}
The open subset
 $\Omega(t) \subset \mathbb{R}^3  $ denotes the changing volume occupied by the fluid,  $\Gamma(t):= \partial\Omega(t)$ denotes
 the moving free-surface, $ \mathcal{V} (\Gamma(t))$ denotes normal
 velocity of $\Gamma(t)$, and $n(t)$ denotes the exterior unit normal vector to the free-surface  $\Gamma(t)$.
  The vector-field $u = (u_1,u_2,u_3)$ denotes the Eulerian velocity
field, and $p$ denotes the pressure function. We use the notation $D=(\p_1,\p_2,\p_3)$ to denote the gradient operator. We have
normalized the equations to have all physical constants equal to 1.

This is a free-boundary partial differential equation to determine the velocity and pressure in the fluid, as well as the location and
smoothness of the a priori unknown free-surface.   In the case that the fluid is irrotational, $ \operatorname{curl} u =0$, the
coupled system of Euler equations (\ref{euler}) can be reduced to an evolution equation for the free-surface (with potential
flow in the interior), in which case (\ref{euler}) simplifies to the water waves equation.  {\it We do not make any irrotationality assumptions.}

We will prove that  the 3-D Euler equations (\ref{euler}) admit classical solutions which evolve regular initial data onto a state,
at finite-time $T>0$,  at which the free-surface self-intersects, and the flow map loses injectivity.   The self-intersection can occur
at a point, causing  a ``splash,'' or on a surface, creating a ``splat.''

\subsection{Local-in-time well-posedness} We begin with a brief history of the local-in-time existence theory for the free-boundary
incompressible Euler equations.
For the {\it irrotational} case of the water waves problem, and for
2-D fluids (and hence 1-D interfaces), the earliest local existence results
were obtained by Nalimov \cite{Na1974}, Yosihara \cite{Yo1982}, and 
Craig \cite{Cr1985} for initial data near equilibrium.  Beale, Hou, \&
Lowengrub \cite{BeHoLo1993} proved that the linearization  of the 2-D water wave 
problem is well-posed if the Rayleigh-Taylor sign condition
\begin{equation}\label{taylor}
\left.\frac{\p p}{\p n}\right|_{t=0} < 0 \ \ \ \text{ on } \Gamma|_{t=0} \,
\end{equation} 
is satisfied by the initial data (see \cite{Ra1878} and \cite{Taylor1950}).
Wu \cite{Wu1997} established local well-posedness for the 2-D 
water waves problem and showed that, due to irrotationality, the Taylor sign condition is satisfied.  
Later Ambrose \& Masmoudi \cite{AmMa2005}, proved local
well-posedness of the 2-D water waves problem as the limit of zero surface tension.
For 3-D fluids (and  \mbox{2-D} interfaces), Wu \cite{Wu1999} used Clifford analysis to prove local existence of the water waves problem with {\it infinite depth}, again showing that the Rayleigh-Taylor sign 
condition is always satisfied in the irrotational case by virtue of the maximum
principle holding for the potential flow.  Lannes \cite{La2005} provided a proof
for the {\it finite depth case with varying bottom}.     Recently, Alazard, Burq \& Zuily \cite{AlBuZu2012} have established low regularity solutions (below
the Sobolev embedding) for the
water waves equations.

 The first local well-posedness result for
the 3-D incompressible Euler equations without the irrotationality assumption was obtained by 
Lindblad \cite{Li2004}  in the case that the domain is diffeomorphic to the unit ball  using a Nash-Moser iteration.
In Coutand and Shkoller \cite{CoSh2007}, we obtained the local well-posedness result for {\it arbitrary initial geometries} that have $H^3$-class
boundaries and without derivative loss (this framework, employing local coordinate charts in the Lagrangian configuration, is ideally suited for the 
splash  and splat singularity problems that we study herein).      Shatah and Zeng \cite{ShZe2008} established a priori estimates
for this problem
using an infinite-dimensional geometric formulation, and  Zhang and Zhang proved well-poseness by extending
the complex-analytic method of Wu \cite{Wu1999} to allow for vorticity.  Again, in the latter case the domain was with infinite depth.

\subsection{Long-time existence}  It is of great interest to understand if solutions to the
Euler equations can be extended for all time when the data is sufficiently smooth and small, or if a finite-time
singularity can be predicted for other types of initial conditions.  

 Because of irrotationality,  the water waves problem does not suffer from vorticity concentration;  therefore, singularity formation
 involves only the loss of regularity of the interface.  In the case that the irrotational fluid is infinite in the horizontal directions,
 certain dispersive-type properties can be made use of.
 For sufficiently smooth and small data,
 Alvarez-Samaniego and Lannes \cite{AlLa2008} proved existence of solutions to the
water waves problem  on large time-intervals (larger than predicted by energy estimates), 
and provided a rigorous justification for a variety of  asymptotic regimes.   By constructing a transformation to remove
the quadratic nonlinearity, combined with decay estimates for the linearized problem (on the infinite half-space domain), Wu \cite{Wu2009} established  an almost global existence result (existence on time intervals which are exponential in the size of the data) for the 2-D water
waves problem with sufficiently small data.   Wu \cite{Wu2011} then proved global existence in 3-D for small data.   Using
the method of spacetime resonances,  Germain, Masmoudi, and Shatah \cite{GeMaSh2009} also established global existence
for the 3-D irrotational problem for sufficiently small data.    

\subsection{Splashing of liquids and the finite-time splash singularity} 

The study  of splashing, and in particular, of drop impact on liquid surfaces has a long history that goes back to
the end of the last century when Worthington \cite{Wo1894} studied the process by means
of single-flash photography.   Numerical studies show both fascinating and unexpected fluid behavior during the splashing
process (see, for example, O\~{g}uz \& Prosperetti \cite{OgPr1990}), with agreement from matched asymptotic analysis by
Howison,  Ockendon,  Oliver,  Purvis and  Smith \cite{HoOcOlPuSm2005}.

 The problem of rigorously  establishing a finite-time singularity for the fluid interface has recently
been explored for the 2-D water waves equations by Castro, C\'{o}rdoba, Fefferman, Gancedo,
and G\'{o}mez-Serrano in \cite{CaCoFeGaGo2011, CaCoFeGaGo2011b}, where it was shown that a smooth initial curve exhibits a finite-time
singularity via self-intersection at a point; they refer to this type of singularity as a ``splash'' singularity, and we will 
continue to use this terminology.  (We will give a precise definition of the  splash domain in our  3-D framework 
in Section \ref{subsec:splashdomain}  and we define the splat domain in Section \ref{sec9}.)
\begin{figure}[htbp]
\begin{center}
\includegraphics[scale = 0.4]{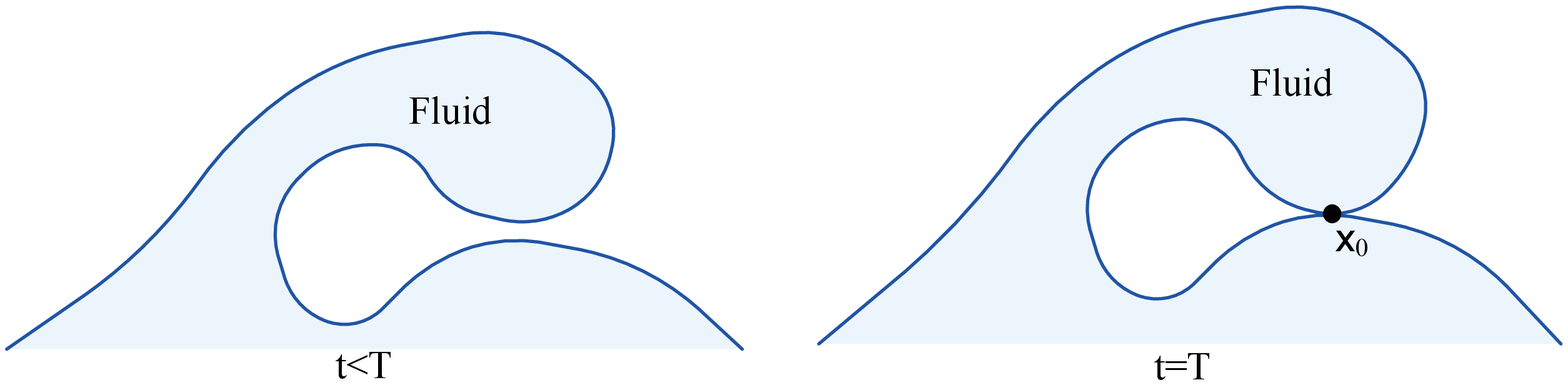}
\caption{The splash singularity wherein the top of the crest touches the trough at a point $x_0$ in finite time
$T$.}
\end{center}
\end{figure}

Their work follows earlier results  by
Castro, C\'{o}rdoba, Fefferman, Gancedo,
and L\'{o}pez-Fern\'{a}ndez \cite{CaCoFeGaLo2011a}  and
Castro, C\'{o}rdoba, Fefferman, Gancedo,
and L\'{o}pez-Fern\'{a}ndez \cite{CaCoFeGaLo2011b} 
for both the Muskat and water waves equations
wherein the authors proved that an initial curve which is graph, that satisfies the Rayleigh-Taylor sign condition,  reaches a regime in finite time
 in which it is no longer a graph and can become unstable due to a reversal of the sign in the Rayleigh-Taylor condition.

Herein,  we develop a new framework for analyzing the finite-time splash and splat singularity for \mbox{3-D} incompressible fluid flows
with vorticity.  Our motivation is to produce a general methodology which can also be applied to compressible fluids, as well
as to ionized fluids, governed by the equations of magnetohydrodynamics.  Our method is founded upon the transformation
of (\ref{euler}) into Lagrangian variables.  We are thus not restricted to potential flows, nor to any special geometries.  Furthermore,
our method of analysis does not, in any significant way, distinguish between flow in different dimensions.   While we present
our results for the case of 3-D fluid flow, they are equally valid in the 2-D case.

\subsection{Main result}
The main result of this paper states that there exist   initial domains $\Omega_0$ of Sobolev class $H^{4.5}$  together with initial
velocity vectors $u_0 \in H^{4}(\Omega_0)$ which satisfy the Rayleigh-Taylor sign condition (\ref{taylor}), such that after a
a finite-time $T>0$ the solution of the Euler equations reaches a  ``splash'' (or ``splat'') singularity.  At such a time $T$,  particles which
were separated at time $t=0$  collide at  a point $x_0$ (or on a surface $\Gamma_0$),   the flow map $\eta(T)$  loses  injectivity, and $\p[\Omega^c]$ forms
a cusp.  In short, $T$ is the time at which the crest of a 3-D wave
turns-over and touches the trough.       This statement is made precise in Theorems \ref{theorem_main} and \ref{theorem2}.

Note that the use of $H^{4.5}$-regularity for the domain $\Omega_0$ and $H^4(\Omega_0)$-regularity for velocity field $u_0$ is
due to the functional framework that we employ for the a priori estimates in Theorem \ref{thm_apriori}.   For 3-D incompressible
fluid flow, we find that this is the most natural functional setting; of course, we could also employ any $H^s$-framework for 
$s \ge 4.5$ or a H\"{o}lder space framework as well.

\subsection{The Lagrangian description}
  We transform the system (\ref{euler}) into
Lagrangian variables.
We let $\eta(x,t)$ denote the ``position'' of the fluid particle $x$ at time $t$.  Thus,
\begin{equation}
\nonumber
\begin{array}{c}
\partial_t \eta = u \circ \eta $ for $ t>0 $ and $
\eta(x,0)=x
\end{array}
\end{equation}
where $\circ $ denotes composition so that
$[u \circ \eta] (x,t):= u(\eta(x,t),t)\,.$
We set
\begin{align*}
v &= u \circ \eta   \text{ (Lagrangian velocity)},  \\
q&=p \circ \eta   \text{ (Lagrangian pressure)}, \\
A &= [D \eta]^{-1}  \text{ (inverse of the deformation tensor)}, \\
J &= \det [D \eta] \text{ (Jacobian determinant of the deformation tensor)}, \\
a &= J\, A \text{ (cofactor of the deformation tensor)}.
\end{align*}
Whenever $ \operatorname{div} u=0$, it follows that $\det D \eta =1$, and hence the cofactor matrix of $D \eta $ is
equal to $[D \eta] ^{-1} $, i.e., $a=A$.
Using  Einstein's summation convention, and using
the notation $F,_k$ to denote $\frac{\p F}{ \p x_k}$, the k$th$-partial derivative of $F$ for $k=1,2,3$,
the Lagrangian version of equations (\ref{euler})  is given on
the fixed reference domain $\Omega$ by
\begin{subequations}
\label{lageuler}
\begin{alignat}{2}
\eta(t) &= e + \int_0^t v\ \ && \text{ in } \Omega \times [0,T] \,, \label{lageuler.a0} \\
 v_t + A^T Dq &=0 \ \ && \text{ in } \Omega \times (0,T] \,, \label{lageuler.a} \\
\operatorname{div} _\eta v &=0 \ \ && \text{ in } \Omega \times [0,T] \,,\label{lageuler.b}  \\
q  &=0 \ \ && \text{ on } \Gamma \times [0,T] \,,\label{lageuler.c}  \\
(\eta,v)  &=(e,u_0) \ \  \ \ && \text{ in } \Omega \times \{t=0\} \,, \label{lageuler.e}
\end{alignat}
\end{subequations}
where $e(x)=x$ denotes the identity map on $\Omega$, and where the $i$th-component of $A^T Dq$ is $ A^k_i \, q,_k$.
($A^T$ denotes the transpose of the matrix $A$.)  By definition of the Lagrangian flow $\eta(t)$,  the free-surface is given by
$$
\Gamma(t) = \eta(t)(\Gamma) \,.
$$
We will also use the notation $\eta(t,\Gamma) = \Gamma(t)$, and $\eta(t, \Omega)= \Omega(t)$.
The Lagrangian divergence is defined by $ \operatorname{div} _ \eta v = A^j_i v^i,_j$.  Solutions to (\ref{lageuler}) which are sufficiently smooth to ensure that $\eta(t)$ are diffeomorphisms,  give solutions to (\ref{euler}) via the change of variables indicated above.

\subsection{The splash singularity for other hyperbolic PDEs}
Our methodology can be applied to a host of other time-reversible PDEs that have a local well-posedness theorem.
\begin{enumerate}
\item {\it Surface tension.}  Our main result also holds if surface tension is added to the Euler equations.  In this case
equation (\ref{lageuler.c}) is replaced with
$$
q n = - \sigma \Delta _g( \eta) \,,
$$
where  $ \sigma >0$ denotes the surface tension parameter, $ n$ is the outward unit-normal to $\Gamma(t)$, $ \Delta _g$ denotes the surface Laplacian with respect to the
induced metric $g$ where $g_{ \alpha \beta } = \eta,_ \alpha  \cdot \eta,_ \beta $.
This is the Lagrangian version of the so-called Laplace-Young boundary condition for pressure: $p= \sigma H$, where
$H$ is the mean curvature of the free-surface $\Gamma(t)$.    We have established well-posedness for this case in
\cite{CoSh2007}.
The only modifications required for the case of positive surface tension
is to consider initial domains $\Omega_0$ of Sobolev class $H^6$ with initial velocity fields $u_0 \in H^{4.5}(\Omega_0)$.   Our main theorem then provides for a finite-time splash singularity for the case that $\sigma >0$.

\item {\it Physical vacuum boundary of a compressible gas.}  We can also consider the evolution of the free-surface compressible
Euler equations which model the expansion of a gas into vacuum.  We established the well-posedness of this system of 
degenerate and characteristic multi-D conservation laws in \cite{CoSh2010b}.    In this setting, our methodology shows that
there exist initial domains $\Omega_0$ of class $H^4$, initial velocity fields $u_0 \in H^{3.5}(\Omega_0)$, and initial density
functions $\rho_0 \in H^4(\Omega_0)$ such after time $T>0$, a splash singularity if formed by the evolving vacuum interface.

\item {\it Other physical models.}   In fact, we can establish existence of a  finite-time splash singularity  for a wide class of hyperbolic systems of PDE which evolve a free-boundary  in a sufficiently smooth functional framework, and
which are locally well-posedness.  Examples of equations (not mentioned above) include nonlinear elasticity
and magnetohydrodynamics.  
\end{enumerate}

%
%

\section{Notation, local coordinates, and some preliminary results} \label{sec:notation} 
\subsection{Notation for the gradient vector} \label{sec:grad-horiz-deriv}

Throughout the paper the symbol $D$ will be used to denote the three-dimensional gradient vector 
$
D=\left( \frac{\p}{\p x_1}\,,  \frac{\p}{\p x_2}\,,  \frac{\p}{\p x_3}  \right)
$.
\subsection{Notation for partial differentiation and Einstein's summation convention} \label{sec:notat-part-diff}

The $k$th partial derivative of $F$ will be denoted by $F\cp{k}=\frac{
\partial F}{
\partial x_k}$. Repeated Latin indices $i,j,k$, etc., are summed from $1$ to $3$, and repeated Greek indices $\alpha, \beta, \gamma$, etc., are summed from $1$ to $2$. For example, $F\cp{ii}=\sum_{i=1}^3\frac{\p^2F}{\p x_i\p x_i}$, and $F^i\cp{\alpha} I^{\alpha\beta} G^i\cp{\beta}=\sum_{i=1}^3\sum_{\alpha=1}^2\sum_{\beta=1}^2\frac{\p F^i}{\p x_\alpha} I^{\alpha\beta} \frac{\p G^i}{\p x_\beta}$.

\subsection{The divergence and curl operators} \label{sec:curl-diverg-oper} For a vector field $u$ on $\Omega$, we set
\begin{align*}
	\Div u & =u^1\cp{1}+u^2\cp{2}+u^3\cp{3} \,, \\
	\curl u& =\Pset{u^3\cp{2}-u^2\cp{3},u^1\cp{3}-u^3\cp{1}, u^2\cp{1} -u^1\cp{2}}. 
\end{align*}
With the permutation symbol $\varepsilon_{ijk}$ given by
{\tiny
$
	\varepsilon_{ijk}= \left\{
	\begin{array}{ll}
		\hfill 1,& \text{even permutation of $\set{1,2,3}$,}\\
		\hfill -1,& \text{odd permutation of $\set{1,2,3}$,}\\
		\hfill 0,&\text{otherwise}\,,
	\end{array}\right.
$}
the $i$th-component of $ \operatorname{curl} u$ is given by
\begin{align*}
	\pset{\curl u}_i=\varepsilon_{ijk} u^k\cp{j}. 
\end{align*}
\subsection{The Lagrangian divergence and curl operators}   We  will write $ \operatorname{div} _\eta v =  \operatorname{div} u \circ 
\eta $ and $ \operatorname{curl} _\eta v = \operatorname{curl} u \circ \eta$.  From the chain rule, 
\begin{align*}
	\Div_\eta v=A_r^sv^r\cp{s} \  \text{ and }  \
	\pset{\curl_\eta v}_i =\varepsilon_{ijk} A_j^s v^k\cp{s} \,.
\end{align*}

%

\subsection{Local coordinates near $\Gamma$ } \label{sec::charts}   
In Appendix A, we establish the a priori estimates for solutions of the 3-D free-surface Euler equations (following
our   local well-posedness theory  in \cite{CoSh2007,CoSh2010}).  Such solutions 
evolve a moving two-dimensional surface which is of Sobolev class
 $H^4$.  This boundary regularity implies a three-dimensional domain of class $H^{4.5}$, constructed via a collection
 of $H^{4.5}$-class local coordinates.

Let $\Omega\subset \mathbb{R}^3  $ denote an open subset of class $H^{4.5}$ and let $\{U_l\}_{l=1}^K$ denote an open covering of $\Gamma=\p\Omega$, such that for each $l\in \{1,2,\dots,K\}$, with 
\begin{align*}
	B&=B(0,1),\text{ denoting the open ball of radius $1$ centered at the origin and}, \\
	B^+&=B\cap\set{x_3>0}, \\
	B^0&=\overline B\cap\set{x_3=0}, 
\end{align*}
there exist $H^{4.5}$-class charts $\thetal$ which satisfy 
\begin{subequations}
\label{normalchart}
\begin{align}
	\thetal\colon B\to\UL\ &\text{ is an $H^{4.5}$ diffeomorphism}, \\
	\thetal(B^+)&=\UL\cap\Omega, \ \ \
	\thetal(B^0)=\UL\cap\Gamma\,.
\end{align}
\end{subequations}
\begin{figure}
	[here] \centering 
	\includegraphics[scale = 0.8]{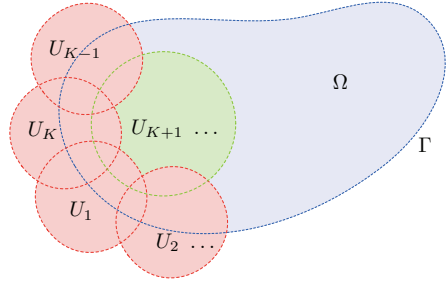}
	\caption{Indexing convention for the open cover $\set{U_\local}_{\local=1}^L$ of $\Omega$.} 
\end{figure}
Next, for $L>K$, we let $\set{\UL}_{\local=K+1}^L$ denote a family of open sets contained in $\Omega$ such that $\set{\UL}_{\local=1}^L$ is an open cover of $\Omega$, and we such that there exist diffeomorphisms $\theta_l:B \to U_l$.

\subsection{Tangential (or horizontal) derivatives}\label{sec: tangential derivative} On each boundary
chart  $\UL\cap\Omega$, for $1\le\local\le K$, we let $\bar \p$ denote the \textit{tangential derivative} whose
 $\alpha$th-component given by
\begin{align*}
	\bar \p_ \alpha  f=\Pset{\frac{\p}{\p x_\alpha}\bset{f\circ\thetal}}\circ\thetal^{-1}=\Pset{\pset{D f\circ\thetal}\frac{\p\thetal}{\p x_\alpha}}\circ\thetal^{-1} \,.
\end{align*}
For functions defined directly on  $B^+$, $\hd$ is simply the horizontal derivative $\hd = (\partial_{x_1}, \partial_{x_2})$.

\subsection{Sobolev spaces} \label{sec:diff-norms-open}

For integers $k\ge0$ and a domain $U$ of $\R^3$, we define the Sobolev space $H^k(U)$ $\pset{H^k(U;\R^3)}$ to be the completion of $C^\infty(\bar{U})$ $\pset{C^\infty(\bar{U};\R^3)}$ in the norm 
\begin{align*}
	\norm{u}_{k,U}^2=\sum_{\abs{a}\le k}\int_U \Abs{D^a u(x) }^2 , 
\end{align*}
for a multi-index $a\in \mathbb{Z}  ^3_+$, with the convention that $\abs{a}=a_1+a_2+a_3$.  When there is no possibility for confusion,
we write $\| \cdot \|_k$ for $\norm{\cdot }_{k,U}$.
For real numbers $s\ge0$, the Sobolev spaces $H^s(U)$ and the norms $\label{n:interior norm}\norm{\cdot}_{s,U}$ are defined by interpolation. 
We will write $H^s(U)$ instead of $H^s(U;\R^3)$ for vector-valued functions.

\subsection{Sobolev spaces on a surface $\Gamma$} \label{sec:sobolev-spaces-gamma} For functions $u\in H^k(\Gamma)$, $k\ge0$, we set 
\begin{align*}
	\abs{u}_{k,\Gamma}^2=\sum_{\abs{a}\le k } \int_\Gamma \Abs{ \hd^a u(x)}^2, 
\end{align*}
for a multi-index $a\in \mathbb{Z}  ^2_+$. For real $s\ge0$, the Hilbert space $H^s(\Gamma)$ and the boundary norm $\label{n:boundary-norm}\abs{\cdot}_s$ is defined by interpolation. The negative-order Sobolev spaces $H^{-s}(\Gamma)$ are defined via duality. That is, for real $s\ge0$, 
$H^{-s}(\Gamma)=H^s(\Gamma)' $.

\subsection{The norm of a standard domain $\Omega$}
\begin{definition} \label{def1}
A domain $\Omega$ is of class $H^{4.5}$ if for each $l=1,...,L$, each diffeomorphism $\theta_l$ is of class $H^{4.5}$.
The $H^{4.5}$-norm of $\Omega$ is defined by
\begin{equation}\label{norm1}
\left( \sum_{l=1}^K \|\theta_l\|^2_{4.5,B^+} + \sum_{l=K+1}^L \|\theta_l\|^2_{4.5,B}\right)^2\,.
\end{equation} 
In particular if $e:\Omega\to\Omega$ is the identity map, then $\|e\|_{4.5, \Omega }$ is given by (\ref{norm1}).
\end{definition} 
We can, of course, replace $H^{4.5}$ with any $H^s$, $s > 2.5$ to define domains $\Omega$ of class $H^s$.

\subsection{Local well-posedness for the free-surface Euler problem}
\begin{theorem}[Coutand and Shkoller \cite{CoSh2007}]\label{thm_local} With $E (t)  $ given by (\ref{Energy}), 
suppose that $E(0) \le M_0$ and that the initial pressure function satisfies the Rayleigh-Taylor sign condition.  Then there
exists  a solution to (\ref{euler}) on $[0,T]$ where $T>0$ depends $E(0)$, and $\sup_{t\in[0,T]} E(t) \le 2M_0$.
Moreover, the solution satisfies 
{\small
$$ \eta \in C([0,T]; H^{4.5}(\Omega))\,, \ v \in C([0,T]; H^{4}(\Omega)) \,, \  \operatorname{curl} _\eta v
C([0,T]; H^{4.5}(\Omega)) \,, \  v_t \in C([0,T]; H^{3.5}(\Omega)) \,.
$$}
\end{theorem} 
\section{The splash domain $\Omega_s$ and its approximation by standard domains $\Omega^ \epsilon $}
\label{sec:splashdomain}
\subsection{The splash domain}

\subsubsection{The meta-definition}
  A {\it splash domain}  $\Omega_s$ is an open  and bounded subset of ${\mathbb R}^n$   which is locally on one side of its boundary, except at a point $x_0\in\partial\Omega_s$, where the domain is locally on each side of the tangent plane at $x_0$.
The domain  $\Omega_s$ satisfies the cone property and can be approximated (in sense to be made precise below) by
domains which  have a smooth boundary.

We observe that the Sobolev spaces $H^r(\Omega_s)$ are defined
for the splash domain  $\Omega_s$ in the same way as for a   domain which is locally on one side of its boundary; moreoever,
as the bounded splash domain $\Omega_s$ satisfies the cone property, interpolation theorems and most of the imporant
Sobolev embedding results hold (see, for examples, Chapters 4 and 5 of Adams~\cite{Adams1978}).
 
%

The main difference between bounded splash domains with the cone property and domains that have the
uniform $H^r$-regularity property  is with regards to  trace theorems:  For the splash domain $\Omega_s$,
a function $f$ in $H^{4.5}(\Omega_s)$ has a trace in $H^4(\Gamma')$ for any smooth subset $\Gamma'$ of $\partial\Omega_s$ whose closure does not contain $x_0$. At $x_0$ there is not a well-defined (global) trace for $f$, in the sense of coming from both sides of the tangent plane
at $x_0$, although  it is indeed possible to define {\it local traces} for $f$ at $x_0$ with respect to each of the  local coordinate charts containing $x_0$.

\subsubsection{The definition of the splash domain}\label{subsec:splashdomain}

\begin{enumerate}
\item
We suppose that $x_0  \in \Gamma:= \partial \Omega_s$ is the unique boundary self-intersection point,
 i.e., $\Omega_s$ is locally on each side of the tangent plane to $\partial\Omega_s=\Gamma_s$ at $x_0$.
 For all other boundary points, the domain is locally on one side of its boundary.   Without loss of
 generality, we suppose that
the tangent plane at $x_0$ is the horizontal plane $x_3-(x_0)_3=0$. 

\item We let $U_0$ denote an open neighborhood of $x_0$ in $ \mathbb{R}  ^3$, and then choose an additional $L$ open
sets $\{U_l\}_{l=1}^L$ such that the collection  $\{U_l\}_{l=0}^K$ is an open cover of $\Gamma_s$, and $\{U_l\}_{l=0}^L$ is an open cover of $\Omega_s$ and such that there exists a
sufficiently small open subset $ \omega \subset U_0$ containing $x_0$ with the property that 
$$\overline\omega \cap \overline{U_l} = \emptyset \ \text{ for all } \ l=1,...,L \,.$$
We set
\begin{align*} 
U_0^+ = U_0 \cap \Omega_s \cap \{ x_3 > (x_0)_3 \} \ \text{ and } U_0^- = U_0 \cap \Omega_s \cap \{ x_3 < (x_0)_3 \} \,.
\end{align*} 
Additionally, we assume that $\overline{U_0}\cap\overline{\Omega_s}\cap\{x_3=(x_0)_3\}=\{x_0\}$, which implies in particular that $U_0^+$ and $U_0^-$ are connected.

\begin{figure}[htbp]
\begin{center}
\includegraphics[scale = 0.4]{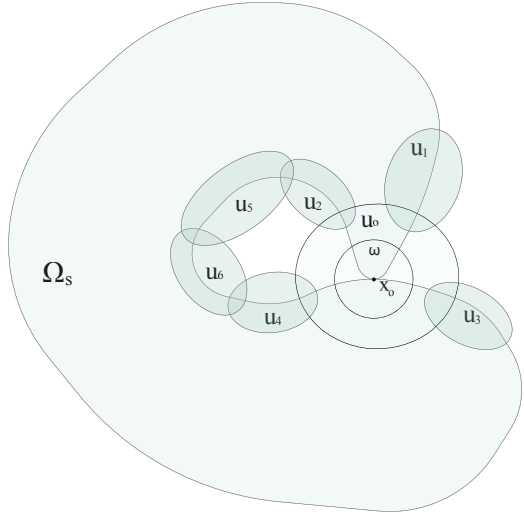}
\caption{Splash domain $\Omega_s$, and the collection of open set $\{U_0,U_1,U_2,...,U_K\}$ covering $\Gamma$.}
\end{center}
\end{figure}

\item For each $l\in \{1,...,K\}$, there exists an  $H^{4.5}$-class diffeomorphism $\theta_l$  satisfying
\begin{gather}
\theta_l : B:=B(0,1) \rightarrow U_l \nonumber \\
U_l \cap \Omega_s = \theta_l ( B^+ )
\ \text{ and } \ \overline{U_l} \cap \Gamma_s = \theta_l ( B^0 ) \,, 
\nonumber
\end{gather}
where 
\begin{align*} 
B^+ &=\{(x_1,x_2,x_3)\in B:  x_3>0\} \,, \\
B^0 &=\{(x_1,x_2,x_3)\in \overline B: x_3=0\}\,.
\end{align*} 

\item For $L > K$, let $\{U_l\}_{l=K+1}^{L}$ denote a family of open sets 
contained in $\Omega_s$ such that 
$\{U_l\}_{l=0}^{L}$ is an open cover of $\Omega_s$, and for $l\in \{K+1,...,L\}$, $\theta_l : B \to U_l$ is an
$H^{4.5}$ diffeormorphism.

\item To the open set $U_0$ we associate two $H^{4.5}$-class diffeomorphisms $\theta_+$ and $\theta_-$ of $B$ onto $U_0$ with the following properties:

\begin{alignat*}{2}
\theta_+(B^+) &= U_0^+ \,,          \qquad \qquad           && \theta_-(B^+)= U_0^-   \,,   \\
\theta_+(B^0) & = \overline{U_0^+}\cap \Gamma_s\,,       &&  \theta_-(B^0) = \overline{U_0^-}\cap \Gamma_s\,,
\end{alignat*}
such that
\begin{equation}\nonumber
\{x_0\}=\theta_+(B^0)\cap\theta_-(B^0)\,,
\end{equation} 
and
\begin{equation}\nonumber
\theta_+(0)=\theta_-(0)=x_0\,.
\end{equation}  
We further assume that 
$$ \overline{\theta_\pm(B^+\cap B(0,1/2))} \cap \overline{\theta_l(B^+)} = \emptyset \text{ for } l=1,...,K \,,$$
and
$$ \overline{\theta_\pm(B^+\cap B(0,1/2))} \cap \overline{\theta_l(B)} = \emptyset \text{ for } l=K+1,...,L \,.$$
\end{enumerate}

 \begin{definition}[Splash domain $\Omega_s$]\label{def:splashdomain} 
 We say that $\Omega_s$ is a splash domain, if it is defined by a collection of
 open covers $\{U_l\}_{l=0}^L$ and associated maps $\{\theta_\pm, \theta_1, \theta_2,...,\theta_L\}$ satisfying the
 properties (1)--(5) above.   Because each of the maps is an $H^{4.5}$ diffeomorphism, we say
that the splash domain $\Omega_s$ defines a self-intersecting {\it generalized} $\bf H^{4.5}$-domain.
 \end{definition} 

\subsection{A sequence of standard domains approximating the splash domain}
We  approximate the two distinguished charts $\theta_-$ and $\theta_+$ 
by charts $\theta_-^\epsilon$ and $\theta_+^\epsilon$ in such a way as to ensure that
$$
\theta_-^ \epsilon (B^0) \cap \theta_+^ \epsilon (B^0) = \emptyset \ \ \forall \ \epsilon >0\,,
$$
and which satisfy
$$
\theta_-^ \epsilon \to \theta_- \ \text{ and } \ \theta_+^ \epsilon \to \theta_+  \ \text{ as } \ \epsilon \to 0 \,.
$$

We choose $\mathfrak{r}>0$ sufficiently small so that
$$
\theta_-(B^+(0, 2\mathfrak{r})) \subset \omega \ \text{ and } \ \theta_+(B^+(0, 2\mathfrak{r})) \subset \omega \,,
$$
and then we let $ \psi \in \mathcal{D} ( B(0, \mathfrak{r}  ))$ denote a smooth bump-function satisfying
 $0\le\psi \le 1$ and $\psi(0)=1$.   For 
$\epsilon>0$ taken small enough, we define 
\begin{align*} 
\theta_-^\epsilon(x)&=\theta_-(x)-\epsilon\  \psi(x)\, \bf e_3\,, \\
\theta_+^\epsilon(x)&=\theta_+(x)+\epsilon\ \psi(x)\, \bf e_3\,,
\end{align*} 
where ${\bf e_3} =(0,0,1)$ denotes the vertical basis vector of the standard basis $\bf e_i$ of $ \mathbb{R}  ^3$.   By 
choosing  $\psi \in \mathcal{D} ( B(0, \mathfrak{r} ))$, we ensure that the modification of the domain is localized to
a small neighborhood of $x_0$ and away from the boundary of $U_0$ and the image of the other maps $\theta_l$.
\begin{figure}[htbp]
\begin{center}
\includegraphics[scale = 0.55]{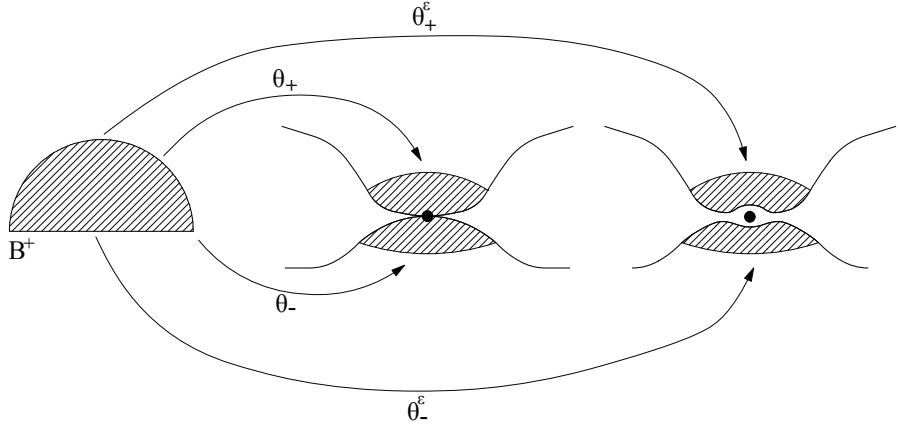}
\caption{The black dot denotes the point $x_0$ where the boundary self-intersects (middle).  For $ \epsilon >0$, the approximate domain $\Omega^ \epsilon $ does not intersect itself (right).}
\end{center}
\end{figure}
 Then, for $\epsilon>0$  sufficiently small,  
 $$
 \theta_-^ \epsilon (\overline{B^+}) \cap  \theta_+^ \epsilon (\overline{B^+})  = \emptyset \,.
 $$
 Since the maps $\theta_\pm^ \epsilon $ are a modification of the maps $\theta_\pm$ in a very small neighborhood
 of $0 \in B$, we have that for $\epsilon>0$  sufficiently small,
 $$
  \theta_\pm^ \epsilon (B^+\cap B(0,1/2)) \cap \theta_l(B^+) = \emptyset \text{ for } l=1,...,K \,,
  $$
and
$$ \theta_\pm^ \epsilon (B^+\cap B(0,1/2)) \cap \theta_l(B) = \emptyset \text{ for } l=K+1,...,L \,.$$
 For $\l \in \{1,...,L\}$ we set $\theta_l^ \epsilon = \theta_l$.     Then $\theta_-^ \epsilon : B^+ \to U_0$, $\theta_+^ \epsilon : B^+ \to U_0$,
 and $\theta_l^ \epsilon : B^+ \to U_l$, $l \in \{1,...,K\}$,  $\theta_l^ \epsilon : B \to U_l$, $l \in \{K+1,...,L\}$, is a collection of $H^{4.5}$ coordinate charts as given in Section \ref{sec::charts},
 and so we have the following
 
 \begin{lemma} [The approximate domains $\Omega^ \epsilon $]\label{approx_domain} For each $ \epsilon >0$ sufficiently
 small,  the set $\Omega^\epsilon$, defined by the local charts $\theta_-^\epsilon:B^+ \to U_0$,  $\theta_+^\epsilon:B^+ \to U_0$,
 and $\theta_l^ \epsilon :B^+ \to U_l$, $l \in \{1,...,K\}$, $\theta_l^ \epsilon :B \to U_l$, $l \in \{K+1,...,L\}$
 is a domain of class $H^{4.5}$, which is locally on one side of its $H^4$ boundary.
 \end{lemma}

By choosing  $0<r_0<{\frac{1}{2}} $ such that $1\ge\psi\ge\frac{1}{2}$ in $B(0,r_0)$, we see that
$$
\left|(\theta^ \epsilon _-(x) -\theta^ \epsilon _+(y))\cdot {\bf e_3}\right|\ge \epsilon\  \text{ for any } \ x,y \in B^+\cap B(0,r_0)\,.
$$
With  $r_0$ chosen, due to the fact that by assumption (2) the images of $\theta_-$ and $\theta_+$ only intersect the plane $\{x_3=(x_0)_3\}$ at the point
$x_0$, there exists $\delta(r_0)>0$ such that $(\theta^ \epsilon_-(x) -x_0)\cdot {\bf e_3}<-\delta(r_0)$ and 
$(\theta^ \epsilon_+(x) -x_0)\cdot {\bf e_3}>\delta(r_0)$ for all $x\in B^+$ with $|x|\ge r_0$. This, in turn,  implies that if 
$x\in B^+$ with $|x|\ge r_0$ and $y\in B^+$, we then have that
$$
\left|(\theta^ \epsilon _-(x) -\theta^ \epsilon _+(y))\cdot {\bf e_3}\right|\ge \delta(r_0)-2\epsilon\ge \epsilon \ \text{  if } \ \epsilon\le \frac{\delta(r_0)}{3} \,.
$$
We have therefore established the following fundamental inequality:  for $0<\epsilon\le  \frac{\delta(r_0)}{3}$,
 \begin{equation}
 \label{verticallocal}
 \forall (x,y)\in B^+\times B^+\,,\ |(\theta^ \epsilon _-(x) -\theta^ \epsilon _+(y))\cdot {\bf e_3} |\ge \epsilon\,.
 \end{equation}
 We henceforth assume that $0<\epsilon\le  \frac{\delta(r_0)}{3}$.

 In summary, we have approximated the self-intersecting splash domain $\Omega_s$ with a sequence of $H^{4.5}$-class 
 domains $\Omega^ \epsilon $,  $0<\epsilon\le  \frac{\delta(r_0)}{3}$ (such that $\p \Omega ^ \epsilon $ does not self-intersect).
  As such, 
 each one of these domains $\Omega^ \epsilon $, $ \epsilon >0$,  will thus be amenable to our local-in-time well-posedness theory 
 for free-boundary incompressible Euler equations with Taylor sign condition satisfied.
 
 We also note that $\Omega^\epsilon$ and $\Omega_s$ are the same domain, except on the two patches $\theta^\epsilon_-(B^+\cap B(0,\frac1 2))$ and $\theta_+^\epsilon(B^+\cap B(0,\frac1 2))$.    In particular,  as $ \theta_ \pm $ differ from $\theta_\pm^ \epsilon $
 on a set properly contained in $\omega \subset U_0$,  we may use the same covering $\{U_l\}_{l=0}^L$ for $\Omega^ \epsilon $
 as for $\Omega_s$.

 \begin{lemma}\label{cslemma1}
  For $0<\epsilon\le  \frac{\delta(r_0)}{3}$, the $H^{4.5}$-norm of $\Omega^ \epsilon $ is bounded independently
 of $ \epsilon $.
\end{lemma} 
\begin{proof} The assertion follows from the following inequality:
$$\|\theta^\epsilon _\pm\|_{4.5,B^+}
\le \|\theta _\pm\|_{4.5,B^+} +  \frac{\delta(r_0)}{3} \|\psi\|_{4.5,B^+} \,.
$$
\end{proof} 
 
\subsection{A uniform cut-off function on the unit-ball $B$}  Let $B_{1- \alpha } = B(0, 1- \alpha )$ for $0 <\alpha < 1$.
  For $ \alpha >0$ taken sufficiently small,  we have that
$\overline{\theta_-( B^+_{1- \alpha } )} \subset U_0$ and $\overline{\theta_+( B^+_{1- \alpha } )} \subset U_0$ and for
each $l=1,...,K$, $\overline{\theta_l( B^+_{1- \alpha } )} \subset U_l$, and for
each $l=K+1,...,L$, $\overline{\theta_l( B_{1- \alpha } )} \subset U_l$, and the open sets  $\theta_-( B^+_{1- \alpha } )$,
$\theta_+( B^+_{1- \alpha } )$,  $\theta_l( B^+_{1- \alpha } )$ ($1\le l\le K$),  $\theta_l( B_{1- \alpha } )$ ($K+1\le l\le L$), are also an open cover of $\Omega_s$.   
Since the  diffeomorphisms $\theta_\pm^ \epsilon$ are modifications for $\theta_\pm$ in a very small neighborhood of
the origin, it is clear that independently of $ \epsilon >0$, the sets   $\theta^\epsilon_-( B^+_{1- \alpha } )$,
$\theta^\epsilon_+( B^+_{1- \alpha } )$,  $\theta_l( B^+_{1- \alpha } )$ ($1\le l\le K$), $\theta_l( B_{1- \alpha } )$ ($K+1\le l\le L$) are also an open cover of each $ \Omega ^ \epsilon $.

\begin{definition}[Uniform cut-off function $\zeta$]\label{zeta}   Let $ \zeta \in \mathcal{D} (B(0,1))$ such that
$0 \le \zeta \le 1$ and $ \zeta(x) =1$ for $|x| <1- \alpha $ and $ \zeta =0 $ for $|x| \ge 1- {\frac{\alpha }{2}} $.

We set $\varsigma = 1- {\frac{\alpha }{2}} $, so that 
\begin{equation}\label{zeta_spt}
0\le  \zeta \in \mathcal{D} (B(0,\varsigma)) \le 1\,.
 \end{equation} 
\end{definition}

\section{Construction of the splash velocity field $u_s$ at the time of the  splash singularity}\label{sec:splashvelocity}

We can now define the so-called {\it splash velocity}  $u_s$ associated with the generalized $\bf H^{4.5}$-class splash domain 
$\Omega_s$, as well as a
sequence of approximations $u_s^ \epsilon $ set on our $H^{4.5}$-class approximations $\Omega^\epsilon$ of the splash domain
$\Omega_s$.

\subsection{The splash velocity $u_s$}
\begin{definition}[Splash velocity $u_s$]  \label{def::us}
A velocity field $u_s$ on an $\bf H^{4.5}$-class splash domain $ \Omega_s$ is called a {\it splash velocity} if it satisfies the following
properties:
\begin{enumerate}
\item $ \zeta  u_s \circ \theta _\pm \in  H^{4.5}(B^+)$, $\zeta  u_s\circ\theta_l\in H^{4.5}(B^+)$ for 
each  $1\le l\le K$ and $u_s\in H^{4.5}(\omega)$ for each $\overline\omega\subset\Omega_s$;
\item  so that under the motion of the fluid, the 
sets $U_0^+$ and $U_0^-$  relatively move towards each other, we require that
 \begin{equation}
\label{touch}
u_s^3 \circ \theta_- >C_-\ ,  \ -u_s^3 \circ \theta_+ >C_+\  \ \text{in}\ B^+ \text { and } C_- + C_+ >0\,,
\end{equation}
where $C_-$ and $C_+$ are  constants.
\end{enumerate}
\end{definition}

\begin{definition}[Splash pressure $p_s$]  \label{def::ps}
A pressure function $p_s$ on an $\bf H^{4.5}$-class splash domain $ \Omega_s$ is called a {\it splash pressure}
associated to the splash velocity $u_s$ if it satisfies the following
properties:
\begin{enumerate}
\item  $p_s$ is the unique solution of
\begin{subequations}
  \label{pressure}
\begin{alignat}{2}
- \Delta   p_s &=-\frac{ \p u_s^i}{\p x_j} \frac{\p u_s^j}{\p x_i}\ \ \   &&\text{in} \ \Omega_s \,, \label{pressure.a}\\
\zeta  p_s\circ \theta_\pm &=0  &&\text{on} \  B^0\,, \label{pressure.b} \\
 \zeta p_s\circ \theta_l &=0  &&\text{on} \  B^0  \  \text{ for } \ l=1,...,K\,; \label{pressure.c}
\end{alignat}
\end{subequations} 
\item the {\it splash pressure} $p_s \in H^{4.5}(\Omega_s)$ and
satisfies the local version of the Rayleigh-Taylor sign condition:
\begin{equation} 
\label{sign}
\frac{\p}{\p x_3} (\zeta  p_s\circ\theta_\pm) > C_{RT}>0 \text{ and }
\frac{\p}{\p x_3} \left( \zeta  p_s \circ \theta_l\right) > C_{RT}> 0 \ \text{ on } \  B^0 \  \text{ for } \ l=1,...,K\,.
\end{equation} 
Note that the outward unit normal to $\p B^+\cap B^0$ points in the direction of $-{\bf e_3}$.
\end{enumerate}
\end{definition} 
\begin{remark} 
As $x_0=\theta_-(0)=\theta_+(0)$, and as $p\circ \theta_-(0) = p\circ \theta_+(0)=0$, 
the conditions (\ref{pressure.b}) and  (\ref{pressure.c}) are equivalent to having the usual
vanishing trace $p=0$ on $\Gamma_s$.  As such, 
 $p\in H^1_0(\Omega_s)\cap H^{4.5}(\Omega_s)$.
\end{remark}

For  property (1) in Definition \ref{def::ps}, we note that $p_s$ is the unique $H_0^1(\Omega_s)$  weak solution of
(\ref{pressure}) guaranteed by the Lax-Milgram theorem in $\Omega_s$.    The usual methods of elliptic regularity theory show that
$\zeta  p_s \circ \theta_\pm$ and
 each $\zeta  p_s\circ\theta_l\in H^{4.5}(B^+)$ for $l=1,..,L$, and thus that $p_s\in H^{4.5}(\Omega_s)$.    (Notice that it is the regularity
 of our charts $\theta_\pm$ and $\theta_l$ which limits the regularity of the splash pressure $p_s$.)
 

As we have defined in property (2) of Definition \ref{def::ps}, 
at the point of self-intersection $x_0$, the gradient $\nabla p_s$ has to be defined from each side of the tangent plane at 
$x_0$;  namely, we can define $\nabla p_s\circ\theta_-$ and $\nabla p_s\circ\theta_+$ on $B^0$, and these two vectors are not equal at the origin  $0$ which is the pre-image of $x_0$ under both $\theta_-$ and $\theta_+$.

It is always possible to choose a splash velocity $u_s$ so that (\ref{sign}) holds.   For example, if we choose $u_s$ to
satisfy $ \operatorname{curl} u_s =0$, then (\ref{sign}) holds according to the maximum principle \cite{Wu1997,Wu1999}.
On the other hand, it is not necessary to choose an irrotational splash velocity, and we will not impose such a constraint.
Essentially, as long as the velocity field induces a positive pressure function, then (\ref{sign}) is satisfied.

\subsection{A sequence of approximations $u_s^ \epsilon $ to the splash velocity}
For $ \epsilon >0$, 
we proceed to construct a sequence of approximations $u_s^ \epsilon : \Omega ^ \epsilon \to \mathbb{R}^3  $ to the velocity field $u_s:\Omega_s \to \mathbb{R}^3  $ in the following way:
\begin{subequations}
\label{usepsilon}
\begin{align}
u_s^\epsilon\circ\theta_l&=u_s\circ\theta_l\,,\ \text{in}\ \ B^+\,,\ \text{for}\ l=1,...,K\,; \\
u_s^\epsilon\circ\theta_l&=u_s\circ\theta_l\,,\ \text{in}\ \ B\,,\ \text{for}\ l=K+1,...,L\,; \\
u_s^\epsilon\circ\theta^\epsilon_-&=u_s\circ\theta_-\,,\ \text{and}\,,\ u_s^\epsilon\circ\theta^\epsilon_+=u_s\circ\theta_+\,,\ \text{in}\ B^+\,.
\end{align}
\end{subequations}
We then have the existence of constants $A>0$, $B>0$ such that
\begin{align}
\|u_s^ \epsilon \|_{4.5,\Omega^\epsilon}&\le A\left(  \| \zeta  u_s^ \epsilon \circ\theta_-^ \epsilon \|_{4,B^+}  +  \| \zeta  u_s^ \epsilon \circ\theta_+^ \epsilon \|_{4,B^+} \right. \nonumber \\
& \qquad \left.
+
\sum_{l=1}^K \| \zeta  u_s^ \epsilon \circ\theta^\epsilon_l\|_{4.5,B^+} +
\sum_{l=K+1}^L \| \zeta  u_s^ \epsilon \circ\theta^\epsilon_l\|_{4.5,B}  \right) \le B \|u_s\|_{4.5,\Omega_s}\,.
\label{est_useps}
\end{align}
We next define the approximate pressure function $p^ \epsilon _s$ in $\Omega^\epsilon$ as the  $H_0^1(\Omega^\epsilon)$ 
weak solution of
\begin{subequations}
  \label{pepsilon}
\begin{alignat}{2}
- \Delta  p^\epsilon_s&= \frac{\p {u^\epsilon_s}^i}{\p x_j} \frac{\p {u^\epsilon_s}^j}{\p x_i}\ \ \   &&\text{in} \ \Omega^\epsilon \,, \label{pepsilon.a}\\
p^\epsilon_s&=0  &&\text{on} \ \partial\Omega^\epsilon \,. \label{pepsilon.b}
\end{alignat}
\end{subequations}
Again, standard elliptic regularity theory  then shows that 
$p^\epsilon_s\in H^{4.5}(\Omega^\epsilon)$. Furthermore, since  $\theta_\pm^\epsilon\rightarrow \theta_\pm$ and
 $\theta_l^\epsilon\rightarrow \theta_l$ in $H^{4.5}(B^+)$, we infer from the definition of $u_s^\epsilon$ in (\ref{usepsilon}) that 
 $ \zeta  p^ \epsilon _s\circ\theta_\pm^\epsilon \rightarrow \zeta  p\circ\theta_\pm$ and
$ \zeta p^ \epsilon _s\circ\theta_l^\epsilon \rightarrow \zeta  p\circ\theta_l$ in $H^{4.5}(B^+)$.  
We may thus conclude
from the pressure condition (\ref{sign})  that we also have, uniformly in $\epsilon>0$ small enough, that
\begin{equation}
\label{signepsilon}
\frac{\p}{\p x_3}( \zeta  p^\epsilon_s\circ\theta^\epsilon_\pm)>\frac{C_{RT}}{2}>0 \text{ and }
\frac{\p}{\p x_3}( \zeta  p^\epsilon_s\circ\theta^\epsilon_l)>\frac{C_{RT}}{2}>0\,,\ \text{on}\ \ B^0 \ \text{ for each } \ 1\le l\le K \,.
\end{equation}

\subsection{Solving the Euler equations backwards-in-time from the final states $\Omega^\epsilon $ and  $u_s^ \epsilon $}  
Because the Euler equations are time-reversible, 
we  can solve the following system of free-boundary Euler equations backward-in-time:
\begin{subequations}
\label{lageuler2}
\begin{alignat}{2}
\eta^ \epsilon (t) &= e + \int_0^t v^ \epsilon \ \ && \text{ in } \Omega ^ \epsilon  \times [-T^ \epsilon ,0] \,, \label{lageuler2.a0} \\
 v^ \epsilon _t + [A^ \epsilon]^T Dq^ \epsilon  &=0 \ \ && \text{ in } \Omega^ \epsilon  \times [-T^ \epsilon ,0) \,, \label{lageuler2.a} \\
\operatorname{div} _{\eta ^ \epsilon } v^ \epsilon  &=0 \ \ && \text{ in } \Omega^ \epsilon  \times [-T^ \epsilon ,0]\,,\label{lageuler2.b}  \\
q^ \epsilon   &=0 \ \ && \text{ on } \Gamma^ \epsilon  \times [-T^ \epsilon ,0] \,,\label{lageuler2.c}  \\
(\eta^ \epsilon ,v^ \epsilon )  &=(e,u_s^ \epsilon ) \ \  \ \ && \text{ in } \Omega^ \epsilon  \times \{t=0\} \,, \label{lageuler2.e}
\end{alignat}
\end{subequations}
where $ A^ \epsilon (x,t) = [ D \eta^ \epsilon (x,t)] ^{-1} $.
Thanks to Lemma \ref{approx_domain},  (\ref{est_useps}), and  (\ref{signepsilon}),  we may apply 
our local well-posedness Theorem \ref{thm_local} for (\ref{lageuler2})  backward-in-time.  This then gives us the existence of 
$T^\epsilon>0$, such that there exists a Lagrangian velocity field 
\begin{equation}\label{veps}
v^\epsilon\in L^\infty (-T^\epsilon,0;H^4(\Omega^\epsilon)),
\end{equation} 
 and a Lagrangian flow map 
 \begin{equation}\label{etaeps}
 \eta^\epsilon \in L^\infty (-T^\epsilon,0;H^{4.5}(\Omega^\epsilon)))
 \end{equation} 
 which solve the free-boundary Euler equations (\ref{lageuler2}) with {\it final data}  $u_s^\epsilon$ and {\it final domain} $\Omega^\epsilon$.
  
Denoting the corresponding Eulerian velocity field by
\begin{equation}\label{ueps}
u^ \epsilon = v^ \epsilon \circ {\eta^ \epsilon }^{-1}  \,,
\end{equation} 
it follows that
$u^\epsilon$ is in $ L^\infty ((-T^\epsilon,0);H^4(\Omega^\epsilon(t)))$, where $ \Omega^\epsilon(t) $ denotes the image of $\Omega^ \epsilon $ under the flow map $\eta^ \epsilon (t)$.

In the remainder of the paper we will prove that  the time of existence $T^\epsilon>0$ (for our sequence of backwards-in-time
Euler equations)  is, in fact,  independent of $\epsilon$; that is, $T^\epsilon $ is equal to a time $T>0$, and that 
$\|u^\epsilon(t)\|_{H^4(\Omega^ \epsilon (t))}$ and $\|\eta^\epsilon(t)\|_{H^{4.5}(\Omega^ \epsilon )}$ are bounded on $[-T,0]$ independently of $\epsilon$. This will then provide us with the existence of a solution which culminates in the
{\it splash singularity} $\Omega_s$ at $t=0$, from the initial data 
\begin{align*} 
u_0 & =\lim_{\epsilon\rightarrow 0} u^\epsilon (-T) \,, \\
\Omega_0 & =\lim_{\epsilon\rightarrow 0} \Omega^\epsilon(-T) \,.
\end{align*} 
In particular, when solving the Euler equations forward-in-time from the initial states $\Omega_0$ and $u_0$,
the smooth $H^{4.5}$ domain $\Omega_0$ is dynamically mapped onto the $\bf H^{4.5}$-class splash domain $\Omega_s$
after a time $T$, and the boundary ``splashes onto itself'' creating the self-intersecting 
 {\it splash singularity} at the point $x_0$.

\section{The  main results}\label{sec:maintheorem}

\begin{theorem}[Finite-time splash singularity]\label{theorem_main}  There exist initial domains $\Omega_0$ of class
$H^{4.5}$ and initial velocity fields 
$u_0 \in H^4(\Omega_0)$, which  satisfy the Taylor sign condition
(\ref{taylor}), such that after a finite time $T>0$, the
solution to the Euler equation $\eta(t)$ (with such data) maps $\Omega_0$ onto the 
splash domain $\Omega_s$, satisfying Definition \ref{def:splashdomain}, with final velocity $u_s$.  
This final velocity $u_s$ satisfies the local Taylor sign condition on the
splash domain $\Omega_s$  in the sense of (\ref{sign}).   The splash velocity $u_s$
has a specified relative velocity on the boundary of the splash domain given by (\ref{touch}).
\end{theorem}
The proof of Theorem \ref{sec:maintheorem} is given in Sections \ref{Charts}--\ref{sec8}.   In Sections \ref{sec9}--\ref{sec10}
we define the splat domain $\bf \Omega _s$ and associated splat velocity $\bf u_s$ and establish the following
\begin{theorem}[Finite-time splat singularity]\label{theorem2}  There exist initial domains $\Omega_0$ of class
$H^{4.5}$ and initial velocity fields 
$u_0 \in H^4(\Omega_0)$, which  satisfy the Taylor sign condition
(\ref{taylor}), such that after a finite time $T>0$, the
solution to the Euler equation $\eta(t)$ (with such data) maps $\Omega_0$ onto the 
splat domain $\bf\Omega_s$, satisfying Definition \ref{def:splatdomain}, with final  velocity $\bf u_s$.  
This final splat velocity $\bf u_s$ satisfies the local Taylor sign condition on the
splat domain $\bf \Omega_s$  in the sense of (\ref{sign}).   The splat velocity $\bf u_s$
has a specified relative velocity on the boundary of the splat domain as stated in Definition \ref{def::usplat}.
\end{theorem}

\section{Euler equations set on a finite number of local charts}\label{Charts}
For each $ \epsilon >0$,
the functions $v^ \epsilon $, $\eta^ \epsilon $, and $ u^ \epsilon $, given by (\ref{veps})--(\ref{ueps}), are solutions to
the Euler equations (\ref{lageuler2}) on the time interval $[-T^ \epsilon ,0]$.

For the purpose of obtaining estimates for this sequence of solutions which do not depend on $ \epsilon >0$,  we
pull-back  the Euler equations  (\ref{lageuler2}) set on $\Omega^\epsilon$  by our charts $\theta_\pm^\epsilon$ and $\theta_l^\epsilon$, $l=1,...,L$; in this way we can analyze the 
 equations on the  half-ball $B^+$.  
 
  It is convenient to extend the index $l$ to include both $l=0$ and $l=-1$;  in
particular, we set
 \begin{alignat*}{2}
\theta_{-1}^\epsilon & = \theta_- ^\epsilon  \ \ \text{ and } \ \   &&\theta_{-1} = \theta_- \,,\\
\theta_0^\epsilon & = \theta_+^\epsilon  &&\ \  \theta_0 = \theta_+\,.
\end{alignat*}

Furthermore,  since for $l=-1,0,1,2,...,K$,  the domain of $\theta_l$ is the half-ball $B^+$, and for $l=K+1,...,L$, the domain
of $\theta_l$ is the unit-ball $B$, it is convenient to write
\begin{alignat*}{2}
\theta_-: \mathcal{B}  \to U_0^-\,, \  \theta_+: \mathcal{B} \to U_0^+\,, \
  \theta_l : \mathcal{B} & \to U_l\cap \Omega \ && \text{ for } \  l=1,...,K \,,    \\
     \theta_l : \mathcal{B}&  \to U_l  \  && \text{ for } \  l=K+1,...,L \,,
\end{alignat*} 
so that $ \mathcal{B} $ denotes $B^+$ for $l=-1,0,1,2,...,K$ and $ \mathcal{B} $ denotes $B$ for $l=K+1,...,L$.

The Euler equations, set on $ \mathcal{B} $, then take the following form:
\begin{subequations}
  \label{leuler}
\begin{alignat}{2}
\eta^ \epsilon &=e + \int_0^t v^ \epsilon \ \ \   &&\text{in} \ \Omega^\epsilon \times [-T^\epsilon,0)\,, \label{leuler.a}\\
\partial_t v^ \epsilon \circ\theta_l^\epsilon+ [b_l^\epsilon]^T \,D(q^ \epsilon \circ\theta_l^\epsilon)&=0  &&\text{in} \ \mathcal{B}  \times [-T^\epsilon,0)\,, \label{leuler.b}\\
  \operatorname{div}_{\eta^ \epsilon  \circ \theta_l^\epsilon} v^ \epsilon \circ\theta_l ^\epsilon &= 0     &&\text{in} \   \mathcal{B}  \times [-T^\epsilon,0)
\,, \label{leuler.c}\\
q^ \epsilon \circ\theta_l^\epsilon &= 0 \ \ &&\text{on} \  B_0 \times [-T^\epsilon,0) \,, \label{leuler.d}\\
   (\eta^ \epsilon   \circ \theta_l^\epsilon ,v^ \epsilon \circ\theta_l^\epsilon)  &= (\theta_l^\epsilon, u_s^ \epsilon \circ\theta_l^\epsilon)  \ \ \ && \text{on} \ \mathcal{B} \times\{t=0\} 
                                                 \,, \label{leuler.e} 
\end{alignat}
\end{subequations}
where $[b_l^\epsilon]^T $ denotes the transpose of the matrix $b_l^\epsilon$, and where
for any $l=-1,0,1,2,...L$, 
$b_l ^\epsilon(x,t)  = \left[D (\eta^ \epsilon  (\theta_l^\epsilon(x),t)\right] ^{-1}$.   
For  $l=K+1,...,L$,   the boundary condition (\ref{leuler.d}) is not imposed.

The system (\ref{leuler}) will allow us to analyze  the behavior of $\eta^ \epsilon $, $v^ \epsilon $, and $q^ \epsilon $ in an 
$ \epsilon $-independent fashion.   Fundamental to this analysis is the following
\begin{lemma}[Equivalence-of-norms lemma]
\label{norms}
With the smooth cut-off function $ \zeta $ given in Definition \ref{zeta},
there exist constants $\tilde C_1>0$ and $\tilde C_2>0$ such that for any $\epsilon>0$ and $f\in H^s(\Omega)$ with $0\le s\le 4.5$, 
\begin{equation}
\tilde C_1\sum_{l=-1}^L\| \zeta   f \circ \theta_l^\epsilon\|^2_{s,\mathcal{B} } \le \|f\|^2_{s,\Omega^\epsilon}\le \tilde C_2\sum_{l=-1}^L\| \zeta   f \circ \theta_l^\epsilon\|^2_{s,\mathcal{B} } \,.
\end{equation}
\end{lemma}
\begin{proof}

Since by construction $\|\theta_l^\epsilon\|_{4.5,\mathcal{B} }\le C_l$, the first inequality is obvious.
For the second inequality, we simply notice that with $E =\{x\in \mathcal{B} |\ \zeta(x) = 1\}$,
$\Omega^\epsilon= \cup_{l=-1}^L \theta_l^\epsilon (E)$, 
so that
\begin{align*}
\|f\|_{s,\Omega^\epsilon}&\le C \sum_{l=-1}^L \|f\|_{s,\theta_l^\epsilon(E)}
\le C \sum_{l=-1}^L \|\zeta((\theta^\epsilon_l)^{-1}) f\|_{s,\theta_l^\epsilon(E)}\\
&\le C \sum_{l=-1}^L \|\zeta f(\theta_l^\epsilon)\|_{s,E} \|(\theta_l^\epsilon)^{-1}\|_{4.5,\theta^\epsilon_l(\mathcal{B} )}
\le C \sum_{l=-1}^L \|\zeta f(\theta_l^\epsilon)\|_{s,\mathcal{B} }\,
\end{align*}
where we used the fact that $\text{det}\nabla\theta_l^\epsilon>c_l>0$ for the last inequality.
\end{proof}

\section{Time of existence  $-T$ of solutions to (\ref{lageuler2})  is independent of $\epsilon$}\label{sec7}
Recall that for $ \epsilon >0$,
the functions $v^ \epsilon $, $\eta^ \epsilon $, and $ u^ \epsilon $, given by (\ref{veps})--(\ref{ueps}), are solutions to
the Euler equations (\ref{lageuler2}) on the time interval $[-T^ \epsilon ,0]$.    We now
prove that the time of existence $-T^ \epsilon $ is, in fact, independent of $ \epsilon $.

We begin by using the fundamental theorem of calculus to express the difference between the flow of two particles $x$ and $y$ as
$$\eta^\epsilon(x,t)-\eta^\epsilon(y,t)=x-y+\int_0^t [v^\epsilon(x,s)-v^\epsilon(y,s)]ds\,.$$

Next, for any $x$ and $y$ in $\Omega^\epsilon$ for which we do not have at the same time  $x\in\theta^\epsilon_-(B^+)$ and $y\in\theta^\epsilon_+(B^+)$,
we see that  independently of $\epsilon>0$ small enough,
\begin{equation}\label{cs1}
|\eta^\epsilon(x,t)-\eta^\epsilon(y,t)-(x-y)|\le C_1|t| \sup_{[-T^\epsilon,0]} E^\epsilon(t) \,  |x-y|\,,
\end{equation} 
where we have used the Sobolev embeddding theorem and where 
$$E^ \epsilon (t) = \|\eta^ \epsilon (t)\|^2_{4.5,\Omega^ \epsilon } 
+ \|v^ \epsilon (t)\|^2_{4,\Omega^ \epsilon }  + \| \operatorname{curl} v^ \epsilon (t)\|^2_{3.5,\Omega^ \epsilon }
+  \|v_t^ \epsilon (t)\|^2_{3.5,\Omega^ \epsilon } \,.$$

The inequality (\ref{cs1}) cannot be independent of $ \epsilon >0$ if both 
$x\in\theta^\epsilon_-(B^+)$ and $y\in\theta^\epsilon_+(B^+)$, for in this case, according to (\ref{verticallocal}),
$|x-y| = O( \epsilon)$ as $ \epsilon \to 0$,
whereas $|v^ \epsilon (x,t) - v^ \epsilon (y,t)| = O(1)$ as $ \epsilon \to 0$, and this, in turn,
yields  a global Lipschitz constant for $v^ \epsilon $ of  $O(\frac{1}{\epsilon})$ as $ \epsilon \to 0$.

When $x\in\theta^\epsilon_-(B^+)$ and $y\in\theta^\epsilon_+(B^+)$,  there exist constants $C_-$, $C_+$, and a polynomial
function $P_1$ which are each independent of $ \epsilon $, such that
\begin{align}
|\eta^\epsilon(x,t)-\eta^\epsilon(y,t)|&\ge |(\eta^\epsilon(x,t)-\eta^\epsilon(y,t))\cdot {\bf e_3}|\n\\
&\ge |(x-y)\cdot {\bf e_3} + t [u_s^\epsilon(x)-u_s^\epsilon(y)]\cdot {\bf e_3}|\n\\
&\ \ \ - \underbrace{\bigl|{\bf e_3}\cdot\int_0^t v^\epsilon(x,t')-u_s^\epsilon(x) dt'\bigr|}_{ \mathcal{I} _1}
-\underbrace{\bigl|{\bf e_3}\cdot\int_0^t v^\epsilon(y,t')-u_s^\epsilon(y) dt'\bigr|}_{ \mathcal{I} _2}    \n\\
& \ge \epsilon +  (C_- +C_+)|t| - t^2 P_1(\sup_{[-T^\epsilon,0]} E^\epsilon)\,,  \label{cs2}
\end{align}
where the triangle inequality has been employed together with
 (\ref{touch}) and (\ref{verticallocal}). 
In order to obtain the lower bound on the terms $ \mathcal{I} _1 $ and $ \mathcal{I} _2$, we again use the
fundamental theorem of calculus, and write
$$v^\epsilon(x,t')=u_s^\epsilon (x) +\int_0^{t'} v^\epsilon_t (x,\tau)d\tau \,;$$ 
using the definition of $E^\epsilon$, it follows that
$$\|v^\epsilon(\cdot ,t')-u_s^\epsilon (\cdot )\|_{L^\infty(\Omega^\epsilon)}  \le C |t'| P_1(\sup_{[-T^\epsilon,0]}E^\epsilon)\,.$$
We proceed to show how the two inequalities (\ref{cs1}) and (\ref{cs2})  (together with the fact that $C_->0$ and $C_+>0$) are used to prove that the time $-T$ is independent of $\epsilon$,  the flow map $\eta^\epsilon$ is injective  on $[-T,0]$, and
the a priori estimates for solutions of (\ref{lageuler2}) are independent of $ \epsilon $
on $[-T,0]$.

We first record our  basic polynomial-type a priori estimate,  given in Theorem
\ref{thm_apriori} in the appendix (see also  \cite{CoSh2007}\cite{CoSh2010}); we  find that on $[-T^\epsilon,0]$, 
\begin{equation}
\label{cs3}
\sup_{t\in [-T^\epsilon,0]} E^\epsilon(t) \le M_0^ \epsilon +|t| P_2(\sup_{[-T^\epsilon,0]} E^\epsilon(t))\,,
\end{equation}
where the constant $M_0^ \epsilon  = P(E^ \epsilon (0))$, i.e. the constant $M_0^ \epsilon $ only depends on initial data
(\ref{leuler.e}).   By Lemma \ref{cslemma1} and (\ref{est_useps}), we see that $M_0^ \epsilon$ is bounded by a constant
$M_0$ which is independent of  $ \epsilon $, so that $\sup_{t\in [-T^\epsilon,0]} E^\epsilon(t) \le 2 M_0$.

We therefore see that if we set
\begin{equation}
\label{cs4}
T=\min\left(\frac{1}{4C_1 M_0}, \frac{C_-+C_+}{2P_1(2 M_0)}, \frac{M_0}{2P_2(2M_0)}\right)\,,
\end{equation}
equation  (\ref{cs1}) implies that on $[-T,0]$,
\begin{equation}
\label{cs5}
|\eta^\epsilon(x,t)-\eta^\epsilon(y,t)| \ge \frac 1 2 |x-y| \ \text{ for } \ \ 
(x,y)\in \theta_l^\epsilon(\mathcal{B} )\times\theta_k^\epsilon(\mathcal{B} )\,, (l,k)\notin \{(-1,0),(0,-1)\}\,,
\end{equation}
while equation (\ref{cs2}) shows that on $[-T,0]$,
\begin{equation}
\label{cs6}
|\eta^\epsilon(x,t)-\eta^\epsilon(y,t)| \ge  \epsilon +  (C_- +C_+)\frac{|t|}{2} \ \text{ for all} \  (x,y)\in \theta_-^\epsilon(\mathcal{B} )\times\theta_+^\epsilon(\mathcal{B} )\,.
\end{equation}

We then have from (\ref{cs5}) and (\ref{cs6}) that the domain $\eta^\epsilon(t,\Omega^\epsilon)$ does not self-intersect for each $t\in[-T,0]$ and from (\ref{cs4}) we also have the estimate
\begin{equation}
\label{cs7}
\sup_{t\in [-T,0]} E^\epsilon(t) \le 2 M_0\,.
\end{equation}

Since $T>0$ is independent of $\epsilon$ by (\ref{cs4}), the estimates we have just obtained will permit the use of  weak convergence
to find the  initial domain $\Omega_0$ at $t=-T$ and  the initial velocity field  $u_0$ at $t=-T$, from which the free surface Euler equations, when run forward in time from $t=0$, will produce the self-intersecting splash domain $\Omega_s$ and velocity field $u_s$ at the final time $T>0$.

\section{Asymptotics as $\epsilon\rightarrow 0$ on the time-interval $[-T,0]$}\label{sec8}

\subsection{Construction of the initial domain $\Omega_0$: the asymptotic domain at $t=-T$}\label{sec::8-1} Theorem \ref{appendix_thm2}
provides continuity-in-time, and
 Lemma \ref{norms} together with the estimate (\ref{cs7}) shows that
\begin{equation*}
\sum_{l=-1}^L\|\zeta\ \eta^\epsilon(\theta_l^\epsilon, -T)\|^2_{4.5,\mathcal{B} }\le \frac{2}{C} M_0\,.
\end{equation*}
Weak compactness and Rellich's theorem provide 
the existence of a subsequence (which by abuse of notation we continue to denote by $\eta^\epsilon$) such that
\begin{subequations}
\begin{align}
\eta^\epsilon(\cdot, -T)\circ\theta_l^\epsilon & \rightharpoonup \Theta_l\,,\ \text{as}\ \epsilon\rightarrow 0\,,\ \text{in}\ 
H^{4.5}(\mathcal{B} _ \varsigma)\,,\label{cs8.a}\\
\eta^\epsilon(\cdot, -T)\circ\theta_l^\epsilon & \rightarrow \Theta_l\,,\ \text{as}\ \epsilon\rightarrow 0\,,\ \text{in}\ 
H^{3.5}(\mathcal{B} _\varsigma)\,,\label{cs8.b}
\end{align}
\end{subequations}
where $ \mathcal{B} _\varsigma = \mathcal{B} \cap B(0,\varsigma)$ and
 $ \varsigma$ is given in Definition \ref{zeta}.

We now define $\Omega_0$ as the union of the sets $\Theta_l(\mathcal{B} _ \varsigma )$ ($-1\le l\le L$).
Due to (\ref{cs5}), (\ref{cs6}) and (\ref{cs8.b}),  we have that
\begin{equation}
\label{cs9}
(x,y)\in \mathcal{B}_\varsigma \times \mathcal{B}_\varsigma \,, ((l,k)\notin \{(-1,0),(0,-1)\})\ \ \ |\Theta_l(x)-\Theta_k (y)| \ge \frac 1 2 |\theta_l (x)-\theta_k (y)|\,,
\end{equation}
and from (\ref{cs6}) on $[-T,0]$,
\begin{equation}
\label{cs10}
\forall (x,y)\in \mathcal{B}_\varsigma \times \mathcal{B}_\varsigma \,,\ 
|\Theta_-(x)-\Theta_+(y)|   \ge    (C_- +C_+)\frac{T}{2} \,,
\end{equation}
where $\Theta_- = \Theta_{-1}$ and $\Theta_+=\Theta_0$.  These inequalities show
 that the boundary of $\Omega_0$ does not self-intersect and that $\Omega_0$ is locally on one side of its boundary. Furthermore, setting $k=l$ in (\ref{cs9}), we see that each smooth map $\Theta_l$ is injective, and thus each $\Theta_l(
 \mathcal{B} _ \varsigma )$ is a domain, which implies that $\Omega_0$ is an open set of $\R^3$.     

\begin{lemma} \label{lemma1}
$\Omega_0$ is a connected,  $H^{4.5}$-class domain, which is locally on one side of its boundary.
\end{lemma} 
\begin{proof}  
{\it Step 1.}
We begin by proving that $\Omega_0$ is connected.
To this end,  fix $X$ and $Y$ in $\Omega_0$ so that 
 $X\in \Theta_l(\mathcal{B}_\varsigma )$ and $Y\in \Theta_j(\mathcal{B}_\varsigma )$ ($-1\le l,j \le L$). 
We let $(x,y)\in \mathcal{B}_\varsigma \times \mathcal{B}_\varsigma $ be such that $X=\Theta_l(x)$ and $Y=\Theta_j(y)$, and we define
\begin{align*} 
 X^\epsilon& =\eta^\epsilon(\theta_l^\epsilon(x),-T)\in \eta^\epsilon(\Omega^\epsilon,-T) \\
 Y^\epsilon& =\eta^\epsilon(\theta_j^\epsilon(y),-T)\in \eta^\epsilon(\Omega^\epsilon,-T) \,.
 \end{align*} 
For $\beta>0$, we set
$$\Omega^\epsilon_\beta=\{Z\in \Omega^\epsilon\ |\ \text{dist}(Z,\partial\Omega^\epsilon)>\beta\}\,.$$ 
Then for $\beta>0$ small enough, we have that 
$\Omega^\epsilon_\beta$ is connected and $X^\epsilon$ and $Y^\epsilon$ are in $\eta^\epsilon(\Omega^\epsilon_\beta,-T)$.

From (\ref{cs8.b}) we infer that each $\eta^\epsilon(\theta_l^\epsilon,-T)$ uniformly converges to $\Theta_l$ in $\mathcal{B}_\varsigma $;  thus, for $\epsilon>0$ small enough,  we find that
\begin{subequations}
\label{cs11}
\begin{align}
\eta^\epsilon(\Omega^\epsilon_\beta,-T)&\subset \Omega_0\,,\label{cs11.a}\\
X^\epsilon \in \Theta_l(\mathcal{B}_\varsigma )\,,\ \ \ &Y^\epsilon \in \Theta_j(\mathcal{B}_\varsigma )\,.\label{cs11.b}
\end{align}
\end{subequations}

Now, as $\Omega^\epsilon_\beta$ is a connected set, so is $\eta^\epsilon(\Omega^\epsilon_\beta,-T)$. 
Since $X^\epsilon$ and $Y^\epsilon$ are in this connected set, we let $C_{X^\epsilon,Y^\epsilon}$ denote a continuous 
path included in 
$\eta^\epsilon(\Omega^\epsilon_\beta,-T)$, and having $X^\epsilon$ and $Y^\epsilon$ as its end-points. 
From (\ref{cs11.a}), $C_{X^\epsilon,Y^\epsilon}\subset \Omega_0$.

Next since both $X$ and $X_\epsilon$ belong to the connected set $\Theta_l(\mathcal{B}_\varsigma )$, let $C_{X,X^\epsilon}$ denote
 a continuous path included in $\Theta_l(\mathcal{B}_\varsigma )\subset\Omega_0$ and having $X$ and $X^\epsilon$ as end-points. 
Similarly, we let $C_{Y^\epsilon, Y}$ denote a continuous path included in $\Theta_j(\mathcal{B}_\varsigma )\subset\Omega_0$ and having $Y^\epsilon$ and $Y$ as its end-points.
We then see that the union of these three paths joins $X$ to $Y$ and is contained in $\Omega_0$, which shows that $\Omega_0$ is connected.

\noindent
{\it Step 2.} 
The fact that
$\Omega_0$ is an $H^{4.5}$-class domain follows immediately  from the convergence given in
 (\ref{cs8.a}).

\noindent
{\it Step 3.}  We conclude by showing that
 $\Omega_0$ is locally on one side of its boundary, and that with $B_\varsigma^0= B^0\cap B(0, \varsigma )$,
\begin{equation}
\label{cs12bis}
\partial\Omega_0=\cup_{l=-1}^K \Theta_l(B_\varsigma^0)\,,
\end{equation}
which will indeed complete the proof that $\Omega_0$ is a standard  $H^{4.5}$-class domain.

To this end we first notice from (\ref{cs8.b}) and the fact that $\eta^\epsilon$ is volume preserving, that for each $l$,
\begin{equation}
\label{cs13bis}
\text{det}\nabla\Theta_l=\text{det}\nabla\theta_l\ge c_l>0\,.
\end{equation}

Also, from (\ref{cs9}) used when $k=l$, we notice that each $\Theta_l$ is an injective map, which with (\ref{cs13bis}) provides
\begin{equation*}
\partial[\Theta_l(\mathcal{B}_\varsigma )]=\Theta_l(\partial \mathcal{B}_\varsigma )\,.
\end{equation*}

Therefore,
\begin{equation}
\label{cs14bis}
\partial\Omega_0\subset \cup_{l=-1}^L \Theta_l(\partial \mathcal{B}_\varsigma )\,.
\end{equation}

Now, let us fix $x\in \partial \mathcal{B}_\varsigma \cap\{x_3>0\}$. We then have (since the only modified charts are modified close to the origin) that for any $-1\le l\le K$,
\begin{equation}
\nonumber
\theta_l^\epsilon(x)=\theta_l(x)\,.
\end{equation}

We also notice that there exists $-1\le k\le L$ and $y\in \mathcal{B}_\varsigma $ such that $\theta_l(x)=\theta_k(y)$ (since $\Omega=\cup_{k=-1}^L\theta_k( \mathcal{B}_\varsigma )$).

We also have that
$\theta_k^\epsilon(y)=\theta_k(y)$, for 
otherwise  $k$ would be equal to either $-1$ or $0$, in which case $\theta_k(y)$ would be  in a very small neighborhood of  $x_0$,
which, in turn, would imply that $l$ must be equal to $k$ (since the charts $\theta_-$ or $\theta_+$ do not intersect the other charts in a small neighborhood of $x_0$), but then we would not be able to
have $x$ at a distance $\varsigma$ from the origin.

We then have $\eta^\epsilon(\theta_l^\epsilon(x), -T)=\eta^\epsilon(\theta_k^\epsilon(y), -T)$ which with (\ref{cs8.b}) implies that
\begin{equation}
\label{cs17bis}
\Theta_l(x)=\Theta_k(y)\in \Theta_k(\mathcal{B}_\varsigma )\subset \Omega_0\,.
\end{equation}

We can prove the same inclusion in a similar way if $x\in \partial \mathcal{B}_\varsigma $ and $K+1\le l\le L$. With (\ref{cs14bis}), this yields
\begin{equation}
\label{cs18bis}
\partial\Omega_0\subset \cup_{l=-1}^K \Theta_l(B_\varsigma^0)\,.
\end{equation}

Now, for $X\in \Theta_l(B_\varsigma^0)$, we have $X=\Theta_l(x)$, with $x\in B_\varsigma^0$. 
Now for any $y\in \overline{ \mathcal{B} _\varsigma}$ such that there exists $-1\le k\le L$ satisfying $X=\Theta_k(y)$,
 we see from (\ref{cs10}) that $(l,k)\notin \{(-1,0),(0,-1)\}$. 
 Therefore, from (\ref{cs9}) we have $\theta_l(x)=\theta_k(y)$.  
 Then, as  $\partial\Omega=\cup_{k=-1}^K \theta_k(B_{\varsigma}^0)$, we see 
 that $y\in B_\varsigma^0$.  Therefore, we have proved that $X$ does not belong to 
 $\cup_{k=-1}^L \Theta_k(\mathcal{B}_\varsigma )=\Omega_0$. Thus $X\in\partial\Omega_0$, which establishes  (\ref{cs12bis}).

 Together with (\ref{cs13bis}) and (\ref{cs10}), this establishes that $\Omega_0$ is a smooth domain locally on one side of its boundary, and concludes the proof.
\end{proof}

\subsection{Asymptotic velocity at $-T$ in the limit $ \epsilon \to 0$}
From our equivalence Lemma \ref{norms} and (\ref{cs7}), 
\begin{equation*}
\sum_{l=-1}^L\|\zeta \ v^\epsilon(\theta_l^\epsilon, -T)\|^2_{4,\mathcal{B} }\le \frac{2M_0}{C}\,,
\end{equation*}
which shows the existence of a subsequence (which we continue to denote  by  the index $\epsilon$) such that
\begin{subequations}
\label{css11}
\begin{align}
v^\epsilon(\cdot,-T)\circ\theta_l^\epsilon & \rightharpoonup V_l\ \ \text{as}\ \epsilon\rightarrow 0\,,\ \text{in}\ H^4(\mathcal{B} _\varsigma)\,,\label{css11.a}\\
v^\epsilon(\cdot,-T)\circ\theta_l^\epsilon & \rightarrow V_l\ \ \text{as}\ \epsilon\rightarrow 0\,,\ \text{in}\ H^3(\mathcal{B} _ \varsigma )\,.\label{css11.b}
\end{align}
\end{subequations}
We now define $u_0$ on $\Omega_0$ as follows:
\begin{equation}
\label{cs12}
\forall l\in \{-1,0,1, 2, ..., L\}\,,\ u_0(\Theta_l)=V_l\ \text{on}\ \mathcal{B}_\varsigma \,.
\end{equation}
In order to justify the definition in (\ref{cs12}), we have to check that if $\Theta_l(x)=\Theta_j(y)$, for $x$ and $y$ in $\mathcal{B}_\varsigma $, then $V_l(x)=V_l(y)$. 
We first notice that if $\Theta_l(x)=\Theta_j(y)$, then by (\ref{cs10}) we have $(l,k)\notin\{(-1,0),(0,-1)\}$. From (\ref{cs9}), we then infer that $\theta_l(x)=\theta_j(y)$ and thus  $|\theta_l^\epsilon(x)-\theta_j^\epsilon(y)|\le c_\epsilon$, with $\lim_{\epsilon\rightarrow 0} c_\epsilon=0$. 

This then, in turn, shows that 
\begin{align*}
|v^\epsilon(\theta_l^\epsilon(x),-T)-v^\epsilon(\theta_j^\epsilon(y),-T)|&\le |\nabla v^\epsilon(\cdot,-T)|_{L^\infty(\Omega^\epsilon(-T))} |\theta_l^\epsilon(x)-\theta_j^\epsilon(y)|\\
& \le c_\epsilon |\nabla v^\epsilon(\cdot,-T)|_{L^\infty(\Omega^\epsilon(-T))}\,,
\end{align*}
which, thanks to (\ref{cs7}), implies that
\begin{equation*}
|v^\epsilon(\theta_l^\epsilon(x),-T)-v^\epsilon(\theta_j^\epsilon(y),-T)|\le  c_\epsilon \sqrt{M_0}\,.
\end{equation*}
By using (\ref{css11.b}), this then implies at the limit $\epsilon\rightarrow 0$:
\begin{equation*}
|V_l(x)-V_j(y)|\le  0\,,
\end{equation*}
which concludes the proof. Also from (\ref{cs12}), we have that $u_0\in H^4(\Omega_0)$, with $\|u_0\|^2_{4,\Omega_0}\le 2 M_0$.

\subsection{Asymptotic domain and velocity on $(-T,0]$ in the limit $ \epsilon \to 0$}\label{sec::asymp}
From our estimate (\ref{cs7}) we then infer the existence of a subsequence (of the subsequence constructed in Section \ref{sec::8-1} and still denoted by a superscript $\epsilon$) such that for all $l=-1,0,1,...,L$
\begin{subequations}
\begin{align}
\partial_t v^\epsilon\circ\theta_l^\epsilon &\rightharpoonup \partial_t v\circ\theta_l\,,\ \text{in}\ L^2(-T,0;H^{3.5}(\mathcal{B}_\varsigma ))\,,  \\
v^\epsilon\circ\theta_l^\epsilon &\rightharpoonup v\circ\theta_l\,,\ \text{in}\ L^2(-T,0;H^4(\mathcal{B}_\varsigma ))\,,
\label{new4-13.1}\\
\eta^\epsilon\circ\theta_l^\epsilon &\rightharpoonup \eta\circ\theta_l\,,\ \text{in}\ L^2(-T,0;H^{4.5}(\mathcal{B}_\varsigma ))\,,
\label{new4-25.1}
\end{align}
\end{subequations}
Next, let $\{\phi_n\}_{n=1}^ \infty $ denote a countable dense set in $H^4 ( \mathcal{B} _\varsigma)$.  We next define the sequence
$f_n^ \epsilon : [-T,0] \to \mathbb{R}  $ by
$$
f_n^ \epsilon (t)  = ( [\eta^\epsilon\circ\theta_l^\epsilon](\cdot,t), \phi_n)_4 \,,
$$
where $(\cdot,\cdot)_4$ denotes the standard inner-product on $H^4 ( \mathcal{B} _\varsigma)$.
Now, for fixed $n$, the uniform bound (\ref{cs7}) together with the fundamental theorem of calculus shows that for a positive constant $\mathfrak{M} < \infty $, 
$\| f_n ^ \epsilon \|_{ C^0 ([-T,0])} \le \mathfrak{M}$ and that $f_n^ \epsilon $ is equicontinuous (as a sequence of functions indexed by the sequence $ \epsilon $).  By the 
Arzela-Ascoli theorem, there exists a subsequence (which we continue to denote by $\epsilon$) such that $f_n^ \epsilon \to f_n$ uniformly on $[-T,0]$. This uniform convergence then implies for all $t\in [-T,0]$ that 
$$\int_0^t f_n^\epsilon(s) ds\rightarrow \int_0^t f_n(s) ds\,.$$
 Due to (\ref{new4-25.1}) we also have (with test function $1_{[0,t]} \phi_n$) that 
 $$\int_0^t f_n^\epsilon(s) ds\rightarrow \int_0^t ([\eta\circ\theta_l](\cdot,t), \phi_n)_4 ds\,,$$
 which by comparison with the previous relation, then shows that
 $$\int_0^t f_n(s) ds= \int_0^t ([\eta\circ\theta_l](\cdot,t), \phi_n)_4 ds\,.$$ Since both integrands are continuous with respect to time, this provides us by differentiation that for all $t\in[-T,0]$,
$$
f_n (t)  = ([\eta\circ\theta_l](\cdot,t), \phi_n)_4  \,.
$$
Next,  since $\{\phi_n\}$ is countable, we may employ the standard diagonal argument to extract a further subsequence (still denoted by $ \epsilon $) such that for all $t\in [-T,0]$,
$$
 ([\eta^\epsilon\circ\theta_l^\epsilon](\cdot,t), \phi)_4  \to   ([\eta\circ\theta_l](\cdot,t), \phi)_4
 $$
for any $\phi \in H^4( \mathcal{B}_\varsigma)$.   This then establishes the existence of a {\it single} subsequence,  such that for all $t\in [-T,0]$, 
\begin{align}
\label{new4-13.2}
\eta^\epsilon\circ\theta_l^\epsilon(\cdot,t)&\rightharpoonup \eta\circ\theta_l(\cdot,t)\,,\ \text{in}\ H^4(\mathcal{B}_\varsigma )\,.
\end{align}

A similar argument shows that for the same subsequence (refined if necessary) and for all $t \in [-T,0]$ ,
\begin{align}
\label{new4-13.3}
v^\epsilon\circ\theta_l^\epsilon (\cdot,t) &\rightharpoonup v\circ\theta_l(\cdot,t) \ \text{in}\ H^{3.5}(\mathcal{B}_\varsigma )\,.
\end{align}

 Theorem \ref{appendix_thm2}
providing continuity-in-time, together with the estimate (\ref{cs7}), we have that for all $t\in [0,T]$,

\begin{subequations}
\begin{align}
\|v^\epsilon\circ\theta_l^\epsilon (\cdot,t)\|^2_{H^{4}(\mathcal{B}_\varsigma )}\le C M_0\,,\\
\|\eta^\epsilon\circ\theta_l^\epsilon (\cdot,t)\|^2_{H^{4.5}(\mathcal{B}_\varsigma )}\le C M_0\,.
\end{align}
\end{subequations}
Together with (\ref{new4-13.2}), this shows that for all $t\in [-T,0]$, for the {\it same sequences} $\eta^\epsilon$, $ v^ \epsilon $, and $\tilde\theta_l^ \epsilon $ as in (\ref{new4-13.2}) 
and (\ref{new4-13.3}), we have the following convergence (by an argument of uniqueness of the weak limit):
\begin{subequations}
\begin{align}
v^\epsilon\circ\theta_l^\epsilon (\cdot,t) &\rightharpoonup v\circ\theta_l(\cdot,t) \ \text{in}\ H^{4}(\mathcal{B}_\varsigma )\,,\\
\eta^\epsilon\circ\theta_l^\epsilon(\cdot,t)&\rightharpoonup \eta\circ\theta_l(\cdot,t) \ \text{in}\ H^{4.5}(\mathcal{B}_\varsigma )\,.
\end{align}
\end{subequations}

Having established the asymptotic limit as $ \epsilon \to 0$ when $t=-T$, we next consider the time interval $(-T,0)$.
We employ the identical argument for taking the limit as $ \epsilon \to 0$ for
the case that $-T<t<0$ as for the case that $t=-T$, leading to an asymptotic domain $\Omega(t)$ of class $H^{4.5}$ and an
Eulerian velocity field $u(\cdot, t)\in H^4(\Omega(t))$ with $\|u(t,\cdot)\|_{4,\Omega(t)}\le  M_0$.


At time $t=0$, there is a slight difference in the asymptotic limit $ \epsilon \to 0$, in the sense that the limit domain is the 
{\it splash domain}  $\Omega_s$, which is a 
self-intersecting generalized $\bf H^{4.5}$-domain, with
the corresponding  limit velocity field is $u_s\in H^{4.5}(\Omega_s)$. 
This limit simply comes from the fact that 
$\|\theta_i^\epsilon-\theta_i\|_{4.5,B^+}\rightarrow 0$ and $\|u_s^\epsilon\circ\theta_i^\epsilon-u_s\circ\theta_i\|_{4.5,B^+}\rightarrow 0$ as $\epsilon\rightarrow 0$.




 
\subsection{Asymptotic Euler equations}
It remains for us to prove that 
$$u_f(x,t)=u(x,t-T)\,, \ \ \ 0\le t\le T$$
 is indeed a solution of the free-surface Euler equations on the moving domain
$$\Omega_f(t)=\Omega(t-T)\,,$$
which evolves
the initial velocity $u_0$ and initial domain $\Omega_0$ onto the final data at time $t=T$
given by  $u_s$ and $\Omega_s$.  This will, in turn,  establish the fact that after a finite time $T$,  the free-surface of the 3-D Euler equations develops a splash singularity.   

We again consider the asymptotic limit as $ \epsilon \to 0$.
For each $\epsilon>0$ fixed, we solve the Euler equations forward-in-time using as initial data,
$\Omega^\epsilon (-T)$ for the domain,  and $u^\epsilon (\cdot,-T)$ for the initial velocity.

To this end, we first define the forward in time quantities for $0\le t\le T$ by
\begin{equation}
\nonumber
\Omega_f^\epsilon(t)=\Omega^\epsilon(t-T)\,,
\end{equation}
and
\begin{subequations}
\begin{alignat*}{2}
u_f^\epsilon (\cdot , t)&=u^\epsilon(\cdot , t-T) &&  \text{in}\ \Omega_f^\epsilon(t)\,,\\
\eta_f^\epsilon(\cdot , t)&=\eta^\epsilon(\cdot , t-T)\circ\eta^\epsilon(\cdot , -T)^{-1} &&  \text{in}\ \Omega_f^\epsilon(0)\,,\\
v_f^\epsilon(\cdot , t)&=v^\epsilon(\cdot , t-T)\circ\eta^\epsilon(\cdot , -T)^{-1} &&  \text{in}\ \Omega_f^\epsilon(0)\,,\\
p_f^\epsilon( \cdot , t)&=p^\epsilon(\cdot , t-T) &&  \text{in}\ \Omega_f^\epsilon(t)\,,\\
q_f^\epsilon( \cdot , t)&=q^\epsilon(\cdot , t-T)\circ\eta^\epsilon(\cdot , -T)^{-1} \qquad &&  \text{in}\ \Omega_f^\epsilon(0)\,.
\end{alignat*}
\end{subequations}
It follows that
\begin{subequations}
\begin{alignat*}{2}
\operatorname{div}u_f^\epsilon &=0 && \  \text{in} \ \Omega_f^\epsilon(t)\,,\\
v_f^\epsilon =u_f^\epsilon\circ\eta_f^\epsilon & =\partial_t \eta_f^\epsilon \qquad && \ \text{in}\ \Omega_f^\epsilon(0)\,,\\
\eta_f^\epsilon(\cdot,0)&=e && \ \text{in}\ \Omega_f^\epsilon(0)\,.
\end{alignat*}
\end{subequations}

From the definitions of $v^ \epsilon $, $\eta^ \epsilon $, and  $u^\epsilon$ in  (\ref{veps})--(\ref{ueps}) and by
uniqueness of solutions to (\ref{lageuler2}), we see that $(u_f^ \epsilon ,p_f^ \epsilon )$ is a solution of (\ref{euler}) on $[0,T]$ with initial domain 
$\Omega_f^\epsilon(0)$ and  initial velocity  $u_f^\epsilon(0)$, with the  domain and velocity at time $t=T$ equal to 
$\Omega^\epsilon$ and $u_s^\epsilon$, respectively. 

In order to analyze the limiting behavior of these solutions as $ \epsilon \to 0$, we write the Euler equations in Lagrangian
form on the fixed domain $\mathcal{B}_\varsigma $ by pulling back the equations from the reference domain $\Omega_f^\epsilon(0)$
using  the following local coordinate charts:
\begin{equation}\nonumber
\tilde\theta_l^\epsilon=\eta^\epsilon(\theta_l^\epsilon, -T)\ \text{ for } \ l=-1,0,1,2,...,L \,.
\end{equation}
Denoting  the local inverse-deformation tensor by
 $$\tilde b^ \epsilon _l=[\nabla(\eta_f^\epsilon\circ\tilde\theta_l^\epsilon)]^{-1}\,,
 $$
 for $-1\le l\le K$, solutions of the Euler equations satisfy 
\begin{subequations}
  \label{leulerbis}
\begin{alignat}{2}
\eta_f^\epsilon\circ\tilde\theta_l^\epsilon &=\tilde\theta_l^\epsilon + \int_0^t v_f^\epsilon\circ\tilde\theta_l^\epsilon\ \ \   &&\text{in} \ \mathcal{B}_\varsigma  \times (0,T]\,, \label{leulerbis.a}\\
\partial_t v_f^\epsilon\circ\tilde\theta_l^\epsilon+ [\tilde b_l^\epsilon]^T \,\nabla (q_f^\epsilon\circ\tilde\theta_l^\epsilon)&=0  &&\text{in} \ \mathcal{B}_\varsigma  \times (0,T)\,, \label{leulerbis.b}\\
\operatorname{div} _{\eta_f^\epsilon\circ \tilde\theta_l^\epsilon}
v_f^\epsilon\circ\tilde\theta_l^\epsilon  &= 0     &&\text{in} \   \mathcal{B}_\varsigma  \times (0,T)
\,, \label{leulerbis.c}\\
q_f^ \epsilon \circ\tilde\theta_l^\epsilon &= 0 \ \ &&\text{on} \  B_0 \times (0,T) \,, \label{leulerbis.d}\\
   (\eta_f^\epsilon,v_f^\epsilon)\circ\tilde\theta_l^\epsilon &= (e, u_f^\epsilon(0))\circ\tilde\theta_l^\epsilon  &&\text{on} \ \mathcal{B}_\varsigma \times\{t=0\} \,, \label{leulerbis.e}                                        
\end{alignat}
\end{subequations}
together with
\begin{equation} 
 \eta_f^\epsilon(\Omega_f^\epsilon(0), T)=\Omega^\epsilon\,.\qquad\qquad\qquad\qquad\qquad
\tag{\ref{leulerbis}f}
\end{equation} 
For $l=K+1,...,L$ the same equations are satisfied with the exception of the boundary condition (\ref{leulerbis}d).

Our a priori estimate Theorem \ref{thm_apriori} shows that for each $l=-1,0,1,2,...,L$
$$
\sup_{t \in [0,T]} \left(
\| \eta^ \epsilon _f(t) \circ\tilde\theta_l^\epsilon\|^2_{4.5, \mathcal{B} _ \varsigma } + \| v^ \epsilon _f(t)\circ\tilde\theta_l^\epsilon\|^2_{4, \mathcal{B} _ \varsigma } 
+  \| q^ \epsilon _f(t)\circ\tilde\theta_l^\epsilon\|^2_{4.5, \mathcal{B} _ \varsigma } \right) \le 2 \tilde M_0^ \epsilon \,,
$$
where $\tilde M_0^ \epsilon $ is a constant that depends on  the $H^{4.5}$-norms of $\tilde\theta_l^\epsilon$ and the
$H^4$-norm of $u_f^ \epsilon (0)$.   Thanks to Lemma \ref{lemma1} and the convergence in  (\ref{css11}), we see that $\tilde M_0^ \epsilon$
is bounded by a constant which is independent of $ \epsilon $.
As such, we have the following convergence in two weak topologies and  one strong topology:
\begin{subequations}
\begin{align}
v_f^\epsilon\circ\tilde\theta_l^\epsilon &\rightharpoonup v_f\circ\Theta_l\,,\ \text{in}\ L^2(0,T;H^4(\mathcal{B}_\varsigma ))\,,\\
\eta_f^\epsilon\circ\tilde\theta_l^\epsilon &\rightarrow \eta_f\circ\Theta_l\,,\ \text{in}\ L^2(0,T;H^3(\mathcal{B}_\varsigma ))\,,\\
q_f^\epsilon\circ\tilde\theta_l^\epsilon &\rightharpoonup q_f\circ\Theta_l\,,\ \text{in}\ L^2(0,T;H^{4.5}(\mathcal{B}_\varsigma ))\,,
\end{align}
\end{subequations}
which together with the 
convergence in  (\ref{cs8.b})  shows, in a manner similar as in Section \ref{sec::asymp}, that for $l=-1,0,1,2,...,K$,  the limit as $ \epsilon \to 0$ of the sequence of solutions to 
(\ref{leulerbis})  is indeed a solution of 
\begin{subequations}
  \label{leulerter}
\begin{alignat}{2}
\eta_f\circ\Theta_l &=\Theta_l + \int_0^t v_f\circ\Theta_l\ \ \   &&\text{in} \ \mathcal{B}_\varsigma  \times (0,T]\,, \label{leulerter.a}\\
\partial_t v_f\circ\Theta_l+[\mathfrak{b}  _l]^T\,\nabla (q_f\circ\Theta_l)&=0  &&\text{in} \ \mathcal{B}_\varsigma  \times (0,T)\,, \label{leulerter.b}\\
  \operatorname{div} _{\eta_f\circ\Theta_l}v_f\circ\Theta_l &= 0     &&\text{in} \   \mathcal{B}_\varsigma  \times (0,T)
\,, \label{leulerter.c}\\
q_f\circ\Theta_l &= 0 \ \ &&\text{on} \  B_0 \times (0,T) \,, \label{leulerter.d}\\
   (\eta_f,v_f)\circ\Theta_l &= (e, u_0)\circ\Theta_l  &&\text{on} \ \mathcal{B}_\varsigma \times\{t=0\} 
                                                 \,, \label{leulerter.e}\\ 
 \eta_f(T,\Omega_0)&=\Omega_s\,,                                                 \end{alignat}
\end{subequations}
where $\mathfrak{b}  _l=[\nabla(\eta_f\circ\Theta_l)]^{-1}$, and 
where $v_f$, $q_f$ and $\eta_f$ are the forward in time velocity, pressure and displacement fields. 

A similar system holds for the interior charts $\Theta_l$, with $K+1\le l\le L$, with the exception of the boundary condition (\ref{leulerter.d}). Therefore, since the charts $\Theta_l$ define $\Omega_0$, we have established that
\begin{subequations}
  \label{leulerqua}
\begin{alignat}{2}
\eta_f &=e + \int_0^t v_f\ \ \   &&\text{in} \ \Omega_0 \times (0,T]\,, \label{leulerqua.a}\\
\partial_t v_f+ A_f^T \,\nabla q_f&=0  &&\text{in} \ \Omega_0 \times (0,T)\,, \label{leulerqua.b}\\
  \operatorname{div} _{\eta_f}  v_f  &= 0     &&\text{in} \   \Omega_0 \times (0,T)
\,, \label{leulerqua.c}\\
q_f &= 0 \ \ &&\text{on} \  \partial\Omega_0 \times (0,T) \,, \label{leulerqua.d}\\
   (\eta_f,v_f) &= (e, u_0)  &&\text{on} \ \Omega_0\times\{t=0\} 
                                                 \,, \label{leulerqua.e} \\
                                                 \eta_f(T,\Omega_0)&=\Omega_s\,,
\end{alignat}
\end{subequations}
where the matrix $A_f = [D\eta_f]^{-1} $.
By a return to Eulerian variables this means that $(u_f,p_f)$ is solution of (\ref{euler}) with initial domain and velocity $\Omega_0$ and $u_0$, respectively,  and final domain and velocity at time $t=T$ equal to the splash domain $\Omega_s$ and $u_s$.

\section{The splat domain $\bf\Omega_s$ and its approximation by standard domains $\bf\Omega^ \epsilon $}\label{sec9}
\label{sec:splatdomain}
\subsection{The splat domain} Whereas our splash domain has a boundary which self-intersects a point $x_0$, an
obvious generalization allows to define the so-called {\it splat domain}  $\bf \Omega_s$,  with  boundary $\bf \p \Omega _s$ which self-intersects on an open subset $\Gamma_0$ of $ \bf \p \Omega _s$.
%
%

%
\subsubsection{The definition of the splat domain}

\begin{enumerate}
\item We suppose that $\Gamma_0  \subset {\bf\Gamma_s}:= \bf \partial \Omega_s$ is the unique boundary self-intersection surface,
 i.e., $\bf \Omega_s$ is locally on each side of $\Gamma_0$ for each $x_0\in\Gamma_0$.
 For all other boundary points, the domain is locally on one side of its boundary. 
 We assume the existence of a smooth level set function $\phi\in H^{4.5}(\R^3)$ such that $\Gamma_0\subset \{\phi=0\}$
  
\item We let $U_0$ denote an open neighborhood of $x_0$ in $ \mathbb{R}  ^3$, and then choose an additional $L$ open
sets $\{U_l\}_{l=1}^L$ such that the collection  $\{U_l\}_{l=0}^K$ is an open cover of $\bf\Gamma_s$, and $\{U_l\}_{l=0}^L$ is an open cover of $\bf\Omega_s$ and such that there exists a
sufficiently small open subset $ \omega \subset U_0$ containing $\Gamma_0$ with the property that 
$$\overline{\omega} \cap \overline{U_l} = \emptyset \ \text{ for all } \ l=1,...,L \,.$$
We set
\begin{align*} 
U_0^+ = U_0 \cap {\bf\Omega_s} \cap \{ \phi > 0 \} \ \text{ and } U_0^- = U_0 \cap {\bf \Omega_s} \cap \{ \phi < 0 \} \,.
\end{align*} 
Additionally, we assume that $\overline{U_0}\cap\overline{{\bf\Omega_s}}\cap\{\phi=0\}=\Gamma_0$, which implies in particular that $U_0^+$ and $U_0^-$ are connected.

\begin{figure}[htbp]
\begin{center}
\includegraphics[scale = 0.4]{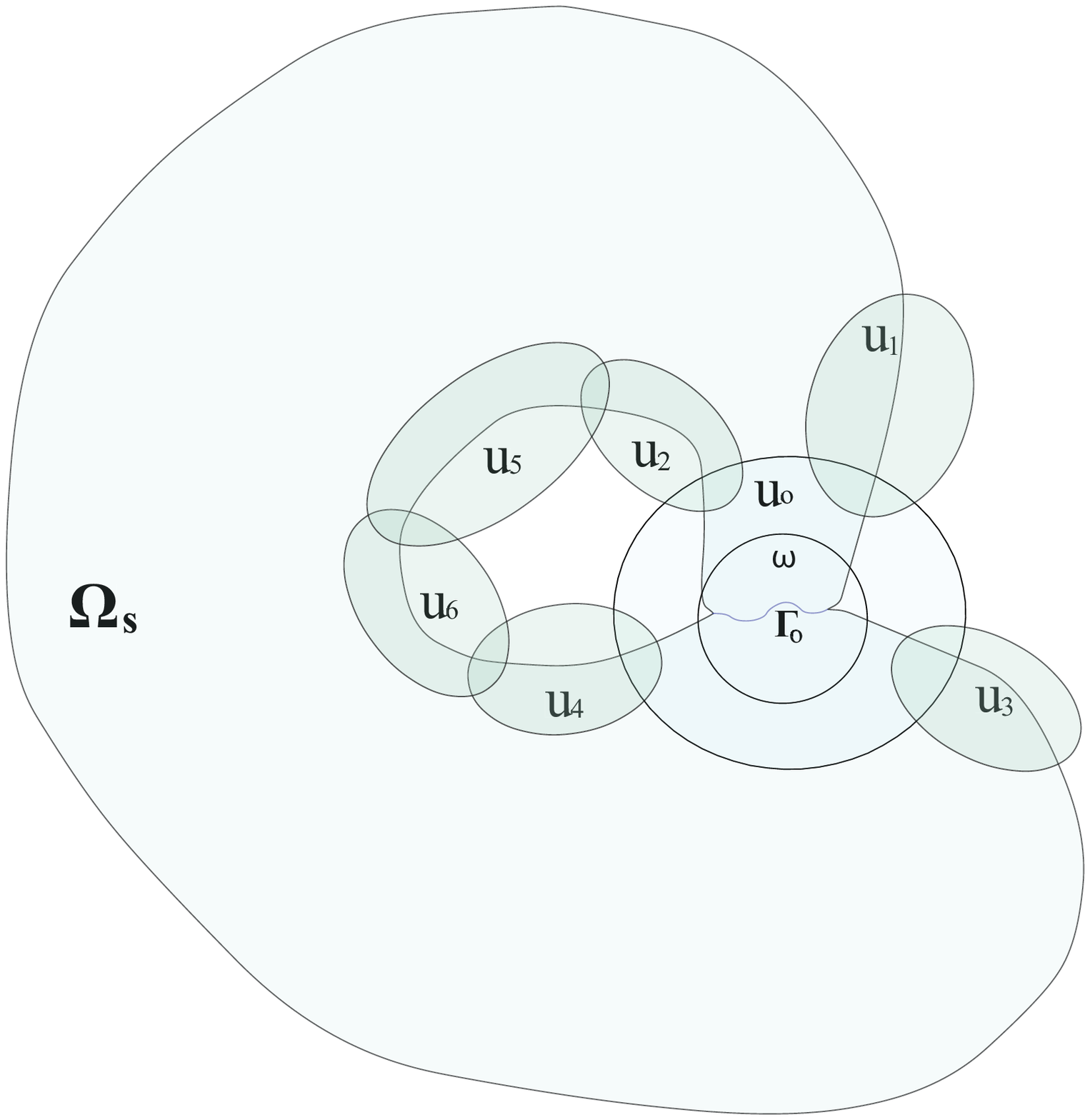}
\caption{Splat domain $\bf\Omega_s$, and the collection of open set $\{U_0,U_1,U_2,...,U_K\}$ covering $\bf \Gamma_s$.}
\end{center}
\end{figure}

\item We furthermore assume that our level set function is such that $\|\nabla\phi\|\ge C_0>0$ on $\overline{U_0}$.

\item For each $l\in \{1,...,K\}$, there exists an  $H^{4.5}$-class diffeomorphism $\theta_l$  satisfying
\begin{gather}
\theta_l : B:=B(0,1) \rightarrow U_l \nonumber \\
U_l \cap {\bf\Omega_s} = \theta_l ( B^+ )
\ \text{ and } \ \overline{U_l} \cap {\bf \Gamma_s} = \theta_l ( B^0 ) \,, 
\nonumber
\end{gather}
where 
\begin{align*} 
B^+ &=\{(x_1,x_2,x_3)\in B:  x_3>0\} \,, \\
B^0 &=\{(x_1,x_2,x_3)\in \overline B: x_3=0\}\,.
\end{align*}

\item For $L > K$, let $\{U_l\}_{l=K+1}^{L}$ denote a family of open sets 
contained in $\bf \Omega_s$ such that 
$\{U_l\}_{l=0}^{L}$ is an open cover of $\bf\Omega_s$, and for $l\in \{K+1,...,L\}$, $\theta_l : B \to U_l$ is an
$H^{4.5}$ diffeormorphism.

\item To the open set $U_0$ we associate two $H^{4.5}$-class diffeomorphisms $\theta_+$ and $\theta_-$ of $B$ onto $U_0$ with the following properties:

\begin{alignat*}{2}
\theta_+(B^+) &= U_0^+ \,,          \qquad \qquad           && \theta_-(B^+)= U_0^-   \,,   \\
\theta_+(B^0) & = \overline{U_0^+}\cap {\bf \Gamma_s}\,,       &&  \theta_-(B^0) = \overline{U_0^-}\cap {\bf \Gamma_s}\,,
\end{alignat*}
such that
\begin{equation}\nonumber
\Gamma_0=\theta_+(B^0)\cap\theta_-(B^0)\,,
\end{equation} 
and
\begin{equation}\nonumber
\theta_+=\theta_-\ \ \text{on}\ \overline\omega_0\subset B_0\,,
\end{equation}  
where $\omega_0$ is a smooth connected domain of $B_0$ in $\R^2$.

We further assume that 
$$ \overline{\theta_\pm(B^+\cap B(0,1/2))} \cap \overline{\theta_l(B^+)} = \emptyset \text{ for } l=1,...,K \,,$$
and
$$ \overline{\theta_\pm(B^+\cap B(0,1/2))} \cap \overline{\theta_l(B)} = \emptyset \text{ for } l=K+1,...,L \,.$$
\end{enumerate}

 \begin{definition}[Splat domain $\bf\Omega_s$]\label{def:splatdomain} 
 We say that $\bf\Omega_s$ is a splat domain, if it is defined by a collection of
 open covers $\{U_l\}_{l=0}^L$ and associated maps $\{\theta_\pm, \theta_1, \theta_2,...,\theta_L\}$ satisfying the
 properties (1)--(6) above.   Because each of the maps is an $H^{4.5}$ diffeomorphism, we say 
that the splat domain $\bf\Omega_s$ defines a self-intersecting generalized $\bf H^{4.5}$-domain.
 \end{definition} 

\subsection{A sequence of standard domains approximating the splat domain}
We  approximate the two distinguished charts $\theta_-$ and $\theta_+$ 
by charts $\theta_-^\epsilon$ and $\theta_+^\epsilon$ in such a way as to ensure that
$$
\theta_-^ \epsilon (B^0) \cap \theta_+^ \epsilon (B^0) = \emptyset \ \ \forall \ \epsilon >0\,,
$$
and which satisfy
$$
\theta_-^ \epsilon \to \theta_- \ \text{ and } \ \theta_+^ \epsilon \to \theta_+  \ \text{ as } \ \epsilon \to 0 \,.
$$

We let $ \psi \in \mathcal{D} (\omega)$ denote a smooth bump-function satisfying
 $0\le\psi \le 1$ and $\psi=1$ on $\Gamma_0$.   For 
$\epsilon>0$ taken small enough, we define the following diffeomorphisms
\begin{align*} 
\theta_-^\epsilon(x)&=\theta_-(x)-\epsilon\  \psi(\theta_-(x))\,  \nabla\phi(\theta_-(x))\,, \\
\theta_+^\epsilon(x)&=\theta_+(x)+\epsilon\ \psi(\theta_+(x))\,  \nabla\phi(\theta_+(x))\,,
\end{align*} 
  By 
choosing  $\psi \in \mathcal{D} (\omega)$, we ensure that the modification of the domain is localized to
a small neighborhood of $\Gamma_0$ and away from the boundary of $U_0$ and the image of the other maps $\theta_l$.
 Then, for $\epsilon>0$  sufficiently small,  thanks to item (3) in the definition of the splat domain,
 $$ \phi(\theta_-^\epsilon(x))\le \phi (\theta_-(x))-\frac{\epsilon}{2} \psi(\theta_-(x))\,  |\nabla\phi(\theta_-(x))|^2<0\,, $$
 $$ \phi(\theta_+^\epsilon(x))\ge \phi (\theta_+(x))+\frac{\epsilon}{2} \psi(\theta_+(x))\,  |\nabla\phi(\theta_+(x))|^2>0\,, $$
 which shows that 
 $$
 \theta_-^ \epsilon (\overline{B^+}) \cap  \theta_+^ \epsilon (\overline{B^+})  = \emptyset \,.
 $$
 Since the maps $\theta_\pm^ \epsilon $ are a modification of the maps $\theta_\pm$ in a very small neighborhood
 of $0 \in B$, we have that for $\epsilon>0$  sufficiently small,
 $$
  \theta_\pm^ \epsilon (B^+\cap B(0,1/2)) \cap \theta_l(B^+) = \emptyset \text{ for } l=1,...,K \,,
  $$
and
$$ \theta_\pm^ \epsilon (B^+\cap B(0,1/2)) \cap \theta_l(B) = \emptyset \text{ for } l=K+1,...,L \,.$$
 For $\l \in \{1,...,L\}$ we set $\theta_l^ \epsilon = \theta_l$.     Then $\theta_-^ \epsilon : B^+ \to U_0$, $\theta_+^ \epsilon : B^+ \to U_0$,
 and $\theta_l^ \epsilon : B^+ \to U_l$, $l \in \{1,...,K\}$,  $\theta_l^ \epsilon : B \to U_l$, $l \in \{K+1,...,L\}$, is a collection of $H^{4.5}$ coordinate charts as given in Section \ref{sec::charts},
 and so we have the following
 
 \begin{lemma} [The approximate domains $\bf\Omega^ \epsilon $]\label{approx_domain2} For each $ \epsilon >0$ sufficiently
 small,  the set $\bf\Omega^\epsilon$, defined by the local charts $\theta_-^\epsilon:B^+ \to U_0$,  $\theta_+^\epsilon:B^+ \to U_0$,
 and $\theta_l^ \epsilon :B^+ \to U_l$, $l \in \{1,...,K\}$, $\theta_l^ \epsilon :B \to U_l$, $l \in \{K+1,...,L\}$ (given in Definition
 \ref{def:splatdomain})
 is a domain of class $H^{4.5}$, which is locally on one side of its $H^4$ boundary.
 \end{lemma} 

 Just as for the splash domain,  we have approximated the self-intersecting {\it splat domain} $\bf\Omega_s$ with a sequence of $H^{4.5}$-class standard
 domains $\bf\Omega^ \epsilon $ locally on one side of its boundary for each $ \epsilon >0$.   
  Also, just as for the splash domain, our approximate domains 
$\bf\Omega^\epsilon$ differ from our splat domain $\bf\Omega_s$ only on the two patches $\theta^\epsilon_-(B^+)$ and $\theta_+^\epsilon(B^+)$.    In particular,  as $ \theta_ \pm $ differ from $\theta_\pm^ \epsilon $
 on a set properly contained in $\omega \subset U_0$,  we continue to use the same covering $\{U_l\}_{l=0}^L$ for $\bf\Omega^ \epsilon $
 as for $\bf\Omega_s$.

\section{Construction of the splat velocity field $u_s$ at the time of the  splat singularity}\label{sec10}

We can now define the {\it splat velocity}  $\bf u_s$ associated with the generalized $\bf H^{4.5}$-class splat domain 
$\bf\Omega_s$, as well as a
sequence of approximations $\bf u_s^ \epsilon $ set on our $H^{4.5}$-class approximations $\bf \Omega^\epsilon$ of the splat domain
$\bf \Omega_s$.

\subsection{The splat velocity $\bf u_s$}
\begin{definition}[Splat velocity $\bf u_s$]  \label{def::usplat}
A velocity field $\bf u_s$ on an $\bf H^{4.5}$-class splat domain $\bf \Omega_s$ is called a {\it splat velocity} if it satisfies the following
properties:
\begin{enumerate}
\item $ \zeta  {\bf u_s} \circ \theta _\pm \in  H^{4.5}(B^+)$, $\zeta {\bf u_s}\circ\theta_l\in H^{4.5}(B^+)$ for 
each  $1\le l\le K$ and ${\bf u_s} \in H^{4.5}(\omega)$ for each $\overline\omega\subset\bf\Omega_s$;
\item ${\bf u_s}\cdot\nabla\phi \circ \theta_-|_{\theta_-(B^+)} >C_-$ and $-{\bf u_s}\cdot\nabla\phi\circ \theta _+|_{\theta_+(B^+)}>C_+$ with $C_-+C_+>0$, so that under the motion of the fluid, the 
sets $U_0^+$ and $U_0^-$ are moving relatively towards each other. 
\end{enumerate}
\end{definition}

We can then define the approximate splat velocity fields $\bf u_s^ \epsilon : {\bf \Omega_s^ \epsilon} \to \mathbb{R}^3  $ in
the same way as we did for the case of the splash velocity. 
The results of Sections \ref{sec7} and \ref{sec8} can then proceed in the same fashion as for the splash case, 
leading to Theorem \ref{theorem2}. 

We note only  that the inequality (\ref{cs2}) must replaced with 
\begin{equation}
-\phi(\eta^\epsilon(x,t))+\phi(\eta^\epsilon(y,t))\ge (C_++C_-)\ |t|-t^2 P(\sup_{[0,t]} E^ \epsilon )\,,   \label{cs1001}
\end{equation}
 for $x$, $y$ as in (\ref{cs2}).   The estimate (\ref{cs1001})
together with
\begin{equation*}
|\phi(\eta^\epsilon(x,t))-\phi(\eta^\epsilon(y,t))|\le |\nabla\phi | \ |\eta^\epsilon(x,t)-\eta^\epsilon(y,t)|\,,
\end{equation*}
and item (3) of the definition of our splat domain ${\bf \Omega_s}$ then provides
\begin{equation}
\label{csplat2}
|\eta^\epsilon(x,t)-\eta^\epsilon(y,t)|\ge \frac{(C_++C_-)}{C_0}\ |t|-t^2 P_1(\sup_{[0,t]} E^ \epsilon )\,.
\end{equation}
This relation is the analogous of (\ref{cs2}) obtained for the approximated splash domain. Since our splat domain is also bounded, we can derive in the same way as for the splash domain a relation similar to (\ref{cs6}) for our approximated splat domain, which shows that $\eta^\epsilon$ is also injective for $\epsilon>0$ small enough.
In turn, this allows us to establish $ \epsilon $-independent estimates and arrive to the analogous conclusions as  those obtained in
Sections \ref{sec7} and \ref{sec8}.

\appendix\label{appendix}
\section{A priori estimates for the free-surface Euler equations}

In this appendix, we establish a priori estimates for the free-surface Euler equations with reference (or initial) domain
$\Omega$ which is a standard $H^{4.5}$-class domain, open, bounded, and locally on one side of its boundary.   

\subsection{Properties of the cofactor matrix $a$,   and a polynomial-type inequality}

\subsubsection{Geometry of the moving surface $\Gamma(t)$} \label{sec:geometry-moving}  With respect to local coordinate charts, 
the vectors $\eta\cp{\alpha}$ for $\alpha=1,2$ span the tangent space to the moving surface $\Gamma(t)=\eta(\Gamma)$ in $\R^3$.
The (\textit{induced}) \textit{surface metric} $g$ on $\Gamma(t)$ has components 
$g_{\alpha\beta}=\eta\cp{\alpha}\cdot\eta\cp{\beta}$.
We let $\label{n:g0}g_0$ denote the surface metric of the initial surface $\Gamma$.
The components of the inverse metric $\label{n:sqrt g}\bset{g}^{-1}$ are denoted by $\bset{g}^{\alpha\beta}$. We use $\sqrt{g}$ to denote $\sqrt{\det g}$; we note that $\sqrt{g}=\abs{\eta\cp{1}\times\eta\cp{2}}$, so that 
\label{n:n}$n(\eta)=\bset{\eta\cp{1}\times\eta\cp{2}}/\sqrt{g}$. 

\subsubsection{Differentiating the inverse matrix $A$}  Using  that $ D\eta \, A = \text{Id}$,
we have the following identities
\begin{align}
{\bar\partial} A^k_i &= - A^s_i {\bar\partial} \eta^r,_s   A^k_r \,, \label{a1} \\
DA^k_i &= - A^s_i D\eta^r,_s   A^k_r \,, \label{a2a} \\
\p_t A^k_i &= - A^s_i v^r,_s   A^k_r \,. \label{a2b}
\end{align}

\subsubsection{Relating the cofactor matrix and the unit normal $n(t)$}  With $N$ denoting the outward unit normal to $\Gamma$,
we have the identity
\begin{align}
n_i( \eta)  &= a^k_i N_k / | a^T N| \,. \nonumber
\end{align}
so that
\begin{equation}\label{n1}
A^k_i N_k =   J ^{-1} \sqrt{g} n_i(\eta) \text{ on } \Gamma \,.
\end{equation}


\subsubsection{A polynomial-type inequality} For a constant $M_0\ge 0$,  suppose that $f(t)\ge 0$,
$t \mapsto f(t)$ is continuous,  and
\begin{equation}\label{f}
f(t) \le M_0 + t\, P(f(t))\,,
\end{equation}
where $P$ denotes a polynomial function.
Then for $t$ taken sufficiently small, we have the bound
$$
f(t) \le 2M_0\,.
$$
We use this type of inequality (see \cite{CoSh2007}) in place of nonlinear Gronwall-type of inequalities.

\subsection{Trace and elliptic estimates for vector fields}
The normal trace theorem  states that the existence of the
normal trace $w \cdot N |_{\Gamma}$ of a velocity field $w\in L^2(\Omega)$ relies on the
regularity of ${\operatorname{div}}  w \in L^2(\Omega) $ (see, for example, \cite{Temam1984}).  If ${\operatorname{div}}  w\in
L^2(\Omega) $, then $w\cdot N$ exists in
$H^{-0.5}(\Gamma)$.  We will use the following variant: 
\begin{align}
|\bar \p w\cdot N|^2_{-0.5, \Gamma} \le C \Big[\|\bar \p w\|^2_{0, \Omega }
+ \|{\operatorname{div}}   w\|^2_{ 0, \Omega ) }\Big] \label{normaltrace}
\end{align}
for some constant $C$ independent of $w$.

The construction of our higher-order energy function is based on the following Hodge-type elliptic estimate:
\begin{proposition}\label{prop1}
For an $H^r$ domain $\Omega$ with $\Gamma = \p \Omega$, $r \ge 3$,
if $F \in L^2(\Omega;{\mathbb R} ^3)$ with $\operatorname{curl}F \in H^{s-1}(\Omega;{\mathbb R} ^3)$,
${\operatorname{div}}F\in H^{s-1}(\Omega)$, and $\bar \p F \cdot N|_{\Gamma} \in
H^{s -{\frac{3}{2}}}(\Gamma)$ for $1 \le s \le r$, then there exists a
constant $\bar C>0$ depending only on $\Omega$ such that
\begin{equation}
\begin{array}{c}
\|F\|_{s, \Omega } \le \bar C\left( \|F\|_{0,\Omega } + \|\operatorname{curl} F\|_{s-1, \Omega }
+ \|\operatorname{div} F\|_{s-1,\Omega } + |\bar \p F \cdot N|_{s-{\frac{3}{2}}, \Gamma }\right)\,,  
\end{array}
\label{hodge}
\end{equation}
where $N$ denotes the outward unit-normal to $\Gamma$.
\end{proposition}
This  well-known inequality  follows from the identity $-\Delta F= {\operatorname{curl}}\,
{\operatorname{curl}}F - D {\operatorname{div}}F$.

\subsection{The higher-order energy function $E(t)$}
\begin{definition} \label{defn_E} We set on $[0,T ]$
\begin{align}
E (t)  = 1+ \|  \eta(t)\|^2_{4.5,\Omega } + \|  v(t)\|^2_{4 , \Omega } + \| \operatorname{curl}_{\eta}  v(t)\|^2_{3.5 , \Omega } + \| v_t(t)\|_{3.5, \Omega }^2
 \,. \label{Energy}
\end{align}
\end{definition}
The function $E(t)$ is the {\it higher-order energy function} which we will prove remains bounded on $[0,T]$.

\begin{definition} We set
 the constant $M_0$ to be a particular polynomial function $P$ of $E(0)$
  so that $M_0 = P ( E (0) ) $.
\end{definition}

\subsubsection{Conventions about constants}\label{subsec_assumptions}  

We take $T  >0$ sufficiently small so that, using the fundamental theorem of calculus, for  constants $c_1,c_2$ and
$t \in [0,T ]$,
\begin{align*}
c_1 \det g (0) &\le \det  g (t)  \le c_2\det   g (0) \text{ on } \Gamma  \,, \\
\| \eta(t)\|_4 & \le\| e\|_4 +1 \,,  \ \
\| q(t)\|_{4}\le \|  q(0) \|_{4} +1 \,,  \\
\| v(t)\|_{3.5} &  \le \| u_0 \|_{3.5} +1 \,,  \  \| v_t(t)\|_{3}   \le \| v_t(0) \|_{3} +1  \,.
\end{align*}
 The right-hand sides appearing in the last three
  inequalities shall be denoted by a
generic constant C in the estimates that we will perform.  The norms are over $\Omega$.

\subsection{Curl and divergence estimates for $\eta$, $v$, and $v_t$}\label{subsec_curlestimates}

\begin{proposition} \label{curl_est}
 For all $t \in (0,T )$,  
\begin{align}
 \|{\operatorname{curl}}\,   \eta(t)\|_{3.5, \Omega }^2
+\|{\operatorname{curl}_{\eta}}\,  v(t)\|_{3.5, \Omega }^2  
\le M_0 +  T\, P({\sup_{t\in[0,T]}} E (t))\,.
\label{curl_estimate}
\end{align}
\end{proposition}
\begin{proof} By taking the curl of (\ref{lageuler.a}), we have that
$$
\operatorname{curl} _{ \eta } v_t =0 \,.
$$
It follows that $\p_t( \operatorname{curl}_{\eta }v) = B(A ,D v)$, where the $k$th-component of $B$ is given by
$$
[B(A ,D v)]_k = \varepsilon_{kji}{ A_t}^s_j v^i,_s  = \varepsilon_{kij} v^i,_s A^s_p\,  v ^p,_l
A^l_j \,;
$$
 hence,
\begin{equation}\label{curlv_3d}
\operatorname{curl}_{\eta }v(t) = \operatorname{curl} u_0  + \int_0^t B(A (t'),Dv(t')) dt'\,.
\end{equation}

\noindent
{\bf Step 1. Estimate for $ \operatorname{curl}   \eta$.}
Computing the gradient of (\ref{curlv_3d}) yields
\begin{equation}\label{cs100}
\operatorname{curl} _{\eta }  Dv(t) = D \operatorname{curl} u_0  - \varepsilon_{ \cdot ji}DA^s_j v^i,_s + \int_0^t DB(A (t'),Dv(t')) dt'\,.
\end{equation} 
(In components, $[\operatorname{curl} _{\eta } \p_{x_l} v]_i = \varepsilon_{ijk} v^k,_{lr} A^r_j$.)
Applying the fundamental theorem of calculus once again, shows that
\begin{align}
\operatorname{curl} _{ \eta } D \eta(t) &= t D \operatorname{curl} u_0  +  \varepsilon_{ \cdot ji}\int_0^t [{A_t}^s_j D\eta^i,_s -DA^s_j v^i,_s] dt' \nonumber\\
& \qquad\qquad + \int_0^t  \int_0^{t'} D B(A (t''),Dv(t'')) dt'' dt' \,, \label{ccss1000}
\end{align}
and finally that
\begin{align}
D \operatorname{curl} \eta(t) &= t D\operatorname{curl} u_0  - \varepsilon_{ \cdot ji}\int_0^t {A_t}^s_j(t') dt' \,  D\eta^i,_s \label{ssscurl1}  \\
&\hspace{-1cm} +  \varepsilon_{ \cdot ji}\int_0^t [{A_t}^s_j D\eta^i,_s -DA^s_j v^i,_s] dt'
+ \int_0^t  \int_0^{t'} DB(A (t''),Dv(t'')) dt'' dt'  \,. \nonumber
\end{align}
Using the fact that $ \p_t A^s_j = - A^s_lv ^l,_p A^p_j$ and $ D A^s_j = - A^s_l D\eta ^l,_p A^p_j$, we see that
%
%
%
\begin{align}
DB(A, Dv) = - \varepsilon_{kji} [Dv^i,_s A^s_l v ^l,_p A^p_j &+ v^i,_s A^s_l Dv ^l,_p A^p_j\nonumber\\
&+ v^i,_s v ^l,_pD( A^s_l  A^p_j)]\,.\label{ccss1010}
\end{align}
The precise structure of the right-hand side  is not very important; rather,
the derivative count is the focus, and as such we write
\begin{align*}
DB( A,  Dv) \sim   D^2 v  \, Dv\,  A \, A
+ D^2 \eta\,  Dv \,  Dv \, A \, A \,.
\end{align*}

Integrating by parts in time in the last term of the right-hand side of (\ref{ssscurl1}), we see that
\begin{align}
\!\int_0^t\!\!\int_0^{t'} DB( A , Dv) \,dt'' dt'    & \sim {-} \int_0^t\!\!\int_0^{t'}
 D^2 \eta  \, (Dv \,  A \, A)_t  
 dt'' dt'   
  +  \int_0^t\!\!\int_0^{t'}  D^2 \eta\,  Dv \,  Dv \, A \, A dt'' dt'  \nonumber \\
 &\qquad + \int_0^t
  D^2 \eta  \, Dv\,  A \, A  dt'  \,. \label{ccss1011}
\end{align}
Thus, we can write
\begin{align*}
D \operatorname{curl} \eta(t) & \sim t D\operatorname{curl} u_0  + {D^2 \eta  \int_0^t D v \, A \, A dt'}
+ \int_0^t D^2 \eta Dv A A dt' \\
& \ \  \int_0^t\int_0^{t'}  D^2 \eta\,  Dv \,  Dv \, A \, A dt'' dt'   + \int_0^t\int_0^{t'}   D^2 \eta  \, (Dv\,  A \, A)_ t dt'' dt'   \,.
\end{align*}

Our goal is to estimate $\| D \operatorname{curl}  \eta\|_{2.5, \Omega }^2$.   Thanks to the Sobolev embedding theorem, we have that
$$
\| D \operatorname{curl}  \eta\|_{2.5,\Omega }^2 \le M_0 +   T\, P(\sup_{t \in [0,T]} E (t)) \,,
$$
and  hence with $ \operatorname{curl} _{ \eta } v_t =0$, that
$$
\| \operatorname{curl}  \eta\|_{3.5, \Omega }^2 \le M_0 +   T\, P(\sup_{t \in [0,T ]} E (t)) \,.
$$

\noindent
{\bf Step 2. Estimate for $ \operatorname{curl}_{\eta}   v$.}   Integrating-by-parts with respect to $\p_t$ in the time integral in
equation (\ref{cs100}), we see that
the highest order term in $ \operatorname{curl} _\eta Dv$ is given by $\int_0^t D^2\eta \, Dv_t\, A\, A dt'$.  
As $H^{2.5}(\Omega)$ is a multiplicative algebra, it follows that on $[0, T]$,
$$
\| \operatorname{curl}_{\eta} v(t) \|_{3.5, \Omega }^2 \le M_0  +   T\, P(\sup_{t \in [0,T]} E (t)) \,.
$$
\end{proof}

\begin{proposition} \label{div_est}
 For all $t \in (0,T )$, 
\begin{align}
 \|{\operatorname{div}}\,   \eta(t)\|_{3.5, \Omega }^2
+\|{\operatorname{div}}\,  v(t)\|_{3, \Omega }^2  \le M_0 +  T\, P({\sup_{t\in[0,T]}} E (t))\,.
\label{div_estimate}
\end{align}
\end{proposition}
\begin{proof}
Since $A^j_i v^i,_j =0$, we see that
\begin{equation}\label{ldiv}
A^j_i Dv^i,_j = - DA^j_i\, v^i,_j  \,.
\end{equation}

\noindent
{\bf Step 1. Estimate for $ \operatorname{div}    \eta$.}  It follows that
$$
[A^j_i D\eta^i,_j]_t  =  \p_t A^j_i  D\eta^i,_j - DA^j_i\, v^i,_j \,.
$$
Using the fact that $\eta (x,0)=x$,
\begin{equation}\label{ccss1012}
[A^j_i D\eta^i,_j] (t)  =   \int_0^t \left( \p_t A^j_i  D\eta^i,_j - DA^j_i\, v^i,_j \right) dt' \,,
\end{equation} 
and hence
$$
D \operatorname{div} \eta(t) =  \int_0^t  \p_t A^j_i  D\eta^i,_j dt'  -
 \int_0^t DA^j_i\, v^i,_j dt'   -  \int_0^t \p_t A^j_i dt' \, D\eta^i,_j  \,.
$$
Again, the Sobolev embedding theorem provides us with the estimate
$$
\| \operatorname{div}  \eta(t)\|^2_{3.5,\Omega }  \le  T\, P( \sup_{t \in [0,T]} E (t)) \,.
$$

\noindent
{\bf Step 2. Estimate for $ \operatorname{div} v$.}   From $A^j_i v^i,_j =0$, we see that
\begin{equation}\label{ssdivv}
\operatorname{div} v(t) = -\int_0^t  \p_t A^j_i dt'\, v^i,_j \,.
\end{equation}
Hence, 
$$
\| \operatorname{div} v(t)\|_{3, \Omega }^2  \le  T\, P( \sup_{t \in [0,T]} E (t)) \,.
$$
\end{proof}

\subsection{Pressure estimates}  Letting $A^j_i \frac{\partial}{\partial x_j}$ act on (\ref{lageuler.a}), for $t\in [0,T ]$, the
Lagrangian  pressure function $q(x,t)$ satisfies
the elliptic equation
\begin{subequations}
\label{qpress}
\begin{alignat}{2}
-\left[ A^j_i A^k_i q,_k\right],_j & = v^i,_j A^j_r \, v ^r,_s \, A^s_i && \text{ in } \ \Omega \,, \\
q & = 0 && \text{ on } \ \Gamma\,.
\end{alignat}
\end{subequations}
Suppose that there exists a weak solution $u \in H^1_0(\Omega) $ to $ - \operatorname{div} [ \mathcal{A} \, D u ]=f$ in $\Omega$ with $u=0$ on
$\Gamma$, and where $ \mathcal{A} $ is positive-definite and symmetric.  Suppose further that $ f \in H^k(\Omega)$, $ \mathcal{A} \in H^{k+1}(\Omega)$ for integers $1 \ge 2$.     Then $u \in H^{k+2}(\Omega)
\cap H^1_0(\Omega) $ and satisfies
\begin{equation}\label{ccss1001}
\|u\|_{k+2} \le C \left( \|f\|_k + \mathcal{P} (\|A\|_{k+1}) \, \|f\|_0\right) \,,
\end{equation} 
where $ \mathcal{P} $ denotes a polynomial function of its argument.
By invoking the Sobolev embedding theorem,  the elliptic estimate (\ref{ccss1001})  shows that
 \begin{align*} 
\|q\|_{4 } &\le C(\| A \|_{2}, \|v\|_{3 }) \,  \| A \|_3  \,,\\
\|q\|_5 &\le C(\| A \|_2, \|v\|_3) \,  \| A \|_4  \,,
\end{align*} 
where the constant has polynomial dependence on $\| A \|_2$ and $\|v\|_3$.  Linear interpolation then yields
$$
\|q\|_{4.5} \le C(\| A \|_2, \|v\|_3) \, \| \eta \|_{4.5} \,.
$$
By time-differentiating  (\ref{qpress}), and
using our conventions of Section \ref{subsec_assumptions} concerning the generic constant $C$, we have the 
elliptic estimate on $ [0,T ]$
\begin{equation}\label{p_est}
\|q(t)\|_{4.5} + \|q_t(t)\|_4 \le C \| \eta(t)\|_{4.5} \,.
\end{equation}
\begin{remark} 
When the elliptic problem  (\ref{qpress}) is set  on the  approximate splash domain $\Omega^ \epsilon $, the elliptic constant
a priori depends on $ \epsilon >0$, via the charts $\theta_\pm ^ \epsilon $; however, thanks to Lemma \ref{cslemma1},
the elliptic constant is independent of $ \epsilon $ since the charts $\theta_\pm$ are bounded in $H^{4.5}$.
 \end{remark} 
 
 \subsection{Rayleigh-Taylor condition at time $t>0$} For each $l=1,...,K$, the fundamental theorem of calculus allows us to write
  $$[q(\theta_l^\epsilon(x),t)],_3=[q(\theta_l^\epsilon(x),0 )],_3+\int_0^t  [q_t(\theta_l^\epsilon(x),t')],_3\ dt'\,,$$ 
From the assumed Rayleigh-Taylor condition (\ref{signepsilon}) {\it on the initial data}, it follows  that for all $x\in B_0$,
\begin{equation*}
 [q(\theta_l^\epsilon(x),t)],_3 \ge \frac{C_{RT}}{2}- C \int_0^t \|q_t(\theta_l^\epsilon,t')\|_3 \ dt'\,.
 \end{equation*}
Thanks to our previously established bound (\ref{p_est}), we then see that on $B_0$,
 \begin{equation}
 \label{RTtime}
 [q(\theta_l^\epsilon(x),t)],_3\ge \frac{C_{RT}}{2}-  t\, P( \sup_{s \in [0,t]} E (s))\,,
 \end{equation}
 so that by choosing $T$ sufficiently small,  $ [q(\theta_l^\epsilon(x),t)],_3\ge \frac{C_{RT}}{4}$ for all $t \in [0,T]$.
 In what follows, we will drop the $\epsilon$ for notational convenience.

\subsection{Technical lemma}
Our energy estimates require the use of the following
\begin{lemma}\label{H12}
Let $H^{\frac{1}{2}} (\Omega)'$ denote the dual
space of $H^{\frac{1}{2}} (\Omega)$.    There exists a positive constant $C$ such that
$$
\| \bar \p F \|_{ H^{\frac{1}{2}}(\Omega)'} \le C\,
\| F \|_{\frac{1}{2}, \Omega }  \ \ \forall F
\in H^{\frac{1}{2}}(\Omega) \,.
$$
\end{lemma}
\begin{proof}
Integrating by parts with respect to the tangential derivative
 yields for all $G\in
H^1(\Omega)$,
\begin{align*}
\int_{\Omega}  \bar \p F \, G \, dx & =  \sum_{l=1}^L \int_{U_l \cap \Omega} \zeta \, [(F \circ \theta_l),_\alpha] \circ \theta_l^{-1} \ G\, dx \\
&=\sum_{l=1}^L \int_{B^+} \zeta \circ \theta_l \, (F \circ \theta_l),_\alpha  \ G \circ \theta_l\,  \det D\theta_l \,dx \\
&=-\sum_{l=1}^L \int_{B^+} \zeta \circ \theta_l \, F \circ \theta_l  \ (G \circ \theta_l),_\alpha\,  \det D\theta_l \,dx  \\
& \qquad\qquad -\sum_{l=1}^L \int_{B^+}  F \circ \theta_l  \ G \circ \theta_l\, (\zeta \circ \theta_l \, \det D\theta_l),_\alpha \,dx \\
&
\le C \| F\|_{0, \Omega }\, \| G \|_{1,\Omega } \,,
\end{align*}
which shows that there exists $C>0$ such that
\begin{equation}\label{H1dual}
\forall F \in L^2(\Omega), \ \ \|  \bar\p F \|_{H^1(\Omega)'} \le C \|F\|_{0,\Omega } \,.
\end{equation}
Interpolating with the obvious inequality
$$
\forall F \in H^1(\Omega), \ \ \| \bar \p F\|_{L^2(\Omega)} \le C \|F\|_{1, \Omega }
$$
 proves the lemma.
\end{proof}

\subsection{Energy estimates for the normal trace of $\eta$ and $v$}
By denoting $\eta_l = \eta \circ \theta_l$ we see that
$$
\eta_l(t): B^+ \to  \Omega(t) \ \text{ for } \ \ l=1,...,K \,.
$$
We set $v_l = u \circ \eta_l$, $q_l = p \circ \eta_l$ and $A_l= [ D \eta_l ] ^{-1} $, $J_l = \det D\eta_l$, and $a_l = J_l A_l$.
It follows that for $l=1,...,K$,
\begin{subequations}
\label{localE}
\begin{alignat}{2}
\eta_l(t) &= \theta_l + \int_0^t v_l\ \ && \text{ in } B^+ \times [0,T] \,, \label{localE.a0} \\
 \p_t v_l  + A_l \, Dq_l &=0 \ \ && \text{ in } B^+ \times (0,T] \,, \label{localE.a} \\
\operatorname{div} _{\eta_l} v_l &=0 \ \ && \text{ in } B^+ \times [0,T] \,,\label{localE.b}  \\
q_l  &=0 \ \ && \text{ on } B^0 \times [0,T] \,,\label{localE.c}  \\
(\eta_l,v_l)  &=(\theta_l,u_0 \circ \theta_l ) \ \  \ \ && \text{ in } B^+ \times \{t=0\} \,. \label{localE.e}
\end{alignat}
\end{subequations}

\begin{proposition}\label{energyest}  For $t\in [0,T ]$,
\begin{align}
|  \p\eta(t)\cdot N |_{3, \Gamma }^2 + | \bar\p v(t)\cdot N |_{2.5, \Gamma }^2   \le  M_0
+  T\, P({\sup_{t\in[0,T]}} E (t))   \,.   \label{main_est}
\end{align}
\end{proposition}
\begin{proof} 
We compute  the following $L^2(B^+)$ inner-product:
\begin{equation}\label{ccss7001}
0= \left( \zeta \bar \p^4 [  \p_t v_l  + A_l \, Dq_l ] \ , \ \zeta \bar \p^4 v_l \right)_{L^2(B^+)} \,.
\end{equation} 
To simplify the notation, we fix $l \in \{1,...,K\}$ and drop the subscript.   We have that
\begin{align} \label{Iintegrals}
0 =  \underbrace{{\frac{1}{2}} \frac{d}{dt} \| \zeta \bar \p^4 v(t)\|^2_{0,B^+}}_{ \mathcal{I} _1 } + \underbrace{ \int_{B^+} \zeta^2 \bar \p^4 A^k_i \ q,_k \, \bar\p^4 v^i dx}_{ \mathcal{I} _2} +  \underbrace{ \int_{B^+} \zeta^2 A^k_i \bar \p^4 q,_k \,
\bar \p^4 v^i dx}_{ \mathcal{I} _3} + \mathcal{R} \,,
\end{align}
where $ \mathcal{R}$ denotes integrals over $B^+$ consisting of lower-order  terms (or remainders)
 which can easily be shown, via the Cauchy-Schwarz inequality, to
satisfy
$$
\int_0^T| \mathcal{R}(t)| dt \le M_0 +  T\, P( \sup_{t \in [0,T]} E(t)) \,.
$$

Using the identity (\ref{a1}), we see that
\begin{align*}
\mathcal{I} _2 & = - \int_{B^+} \zeta^2 A^k_r \bar\p^4 \eta^r,_s A^s_i q,_k \bar \p^4 v^i dx + \mathcal{R}  \\
 &=
 - \underbrace{\int_{B^0} \zeta^2  A^k_r \bar\p^4 \eta ^r  q,_k \bar \p^4 v^i  A^s_i N^0_s \, dx_h}_{ {\mathcal{I} _2}_a}
  +\underbrace{\int_{B^+} \zeta^2 A^k_r \bar\p^4 \eta ^r A^s_i q,_k \bar \p^4 v^i,_s dx }_{ {\mathcal{I} _2}_b}
   + \mathcal{R} \,,
\end{align*}
where $dx_h=dx_1dx_2$ denotes the surface measure on $B^0$.
As $q=0$ on $B^0$,  $q,_1=0$ and $q,_2=0$ on $\Gamma$,  and since the exterior normal on $B_0$ is $N^0=-e_3$, we have $A^3_r=-A^k_r N^0_k$, which then implies
\begin{align*}
  { \mathcal{I} _2}_a
 & =  \int_{B^0}  q,_3 \zeta^2  \bar\p^4 \eta ^r  A^k_r N^0_k  \bar \p^4\ v^i A^s_i N^0_s  \, dx_h \,.
 \end{align*}

We define  $ n_l $ to be the outward unit normal to the moving surface $\eta_l (t, B^0)$,
so that  from (\ref{n1}),
 \begin{equation}\nonumber
{A_l}^k_i N^0_k =   J_l  ^{-1} \sqrt{g_l } n_l (\eta_l) \text{ on } B^0 \,.
\end{equation}

Dropping the subscript $l$ again and writing $n$ for $n(\eta)$,  it follows that
\begin{align*}
  { \mathcal{I} _2}_a(t) & =  \int_{B^0} q,_3 \zeta^2 \bar\p^4  \eta \cdot  n \ \bar \p^4   v \cdot  n  \, |\det  g |  J ^{-2} \, dx_h \\
  &= \frac{1}{2} \frac{d}{dt} \underbrace{ \int_{B^0} q,_3 \zeta^2
   |\bar\p^4  \eta \cdot n |^2 \, |\det g | J^{-2}  \, dx_h }_{ { \mathcal{K}}_a} 
   - \underbrace{ \int_{B^0}{\frac{1}{2}} \zeta^2 \bar\p^4   \eta^i \bar \p^4 
  \eta^j \p_t [ ( n_i n_j \, |\det g| \, J^{-2} ] dx_h}_{{ \mathcal{K}}_b} \,.
 \end{align*}
By the assumption of Section \ref{subsec_assumptions},
$$
\left |  \p_t [ n_in_j\, |\det g | \, J^{-2} ]\right|_{ L^ \infty (\Gamma)} \le C \,,
$$
from which it follows that
$$
\int_0^T { \mathcal{K}}_b(t) dt \le C\,T\, P( \sup_{t \in [0,T]} E (t)) \,.
$$
Using our Rayleigh-Taylor condition (\ref{RTtime}) for  $q,_3(t)$, and bounds for $\det  g (t), J $ which can be established similarly on $[0,T ]$, we see that
$$
\bar c\, |\zeta \bar \p^4  \eta_l(t) \cdot n_l(t)|_{0,B^0}^2 - T\, P( \sup_{t \in [0,T]} E (t)) \le \int_0^T { \mathcal{I} _2}_a(t) dt  \,,
$$
for a constant $\bar c$ which depends on $ C_{RT},  g (0)$,  and $J(0)= \det D\theta_l$.     We set 
$$
N_l = \frac{{\theta_l},_1 \times {\theta_l},_2}{\left|  {\theta_l},_1 \times {\theta_l},_2\right|} \,.
$$

By the fundamental theorem of calculus $n_l(t) = N_l +  \int_0^t \p_t  n_ l(t') dt'$, and by our assumptions in Section \ref{subsec_assumptions},
$\sup_{[0,T]}\left|  \p_t  n_l (t)\right|_{ L^ \infty (\Gamma)} \le C$; hence,
$$
\bar c\, |\zeta \bar \p^4  \eta_l(t) \cdot  N_l |_{0,B^0}^2 \le  \int_0^T { \mathcal{I} _2}_a(t) dt +  T\, P( \sup_{t \in [0,T]} E (t)) \,,
$$
and hence
$$
\bar c\, |\zeta \bar \p  \eta_l(t) \cdot  N_l |_{3,B^0}^2 \le \int_0^T { \mathcal{I} _2}_a(t) dt+  T\, P( \sup_{t \in [0,T]} E (t)) \,.
$$

It remains to show that the integrals $\int_0^T { \mathcal{I} _2}_b(t)dt$ and $\int_0^T \mathcal{I} _3(t)dt$ are both bounded by $ T\, P( \sup_{t \in [0,T]}
E (t))$.    Using (\ref{localE.b}),
\begin{align*}
{ \mathcal{I} _2}_b(t) & =
- \int_{B^+} \zeta^2 A^k_r \, \bar\p^4 \eta ^r\, q,_k \,v^i,_s \, \bar \p^4  A^s_i \,  dx + \mathcal{R}  \\
& \le C \| \zeta \bar \p^4\eta (t)\|_{\frac{1}{2}, B^+} \|\zeta \bar \p^4 A (t)\|_{H^ {\frac{1}{2}} (B^+)'}  + \mathcal{R}  \\
& \le C \|  \bar \p^4 \eta (t)\|_{\frac{1}{2}} \| \bar \p^3 A (t)\|_{H^ {\frac{1}{2}} (\Omega)} + \mathcal{R}  \\
& \le C\sup_{t \in [0,T]} E(t)  + \mathcal{R}   \,,
\end{align*}
where we have used Lemma \ref{H12} for the second inequality.

Finally,
\begin{align*}
\mathcal{I} _3(t) &= - \int_{B^+} \zeta^2 \bar \p^4 q \, \bar \p^4 v^i,_k \, A^k_i\,  dx =
\int_{B^+} \zeta^2\bar \p^4 q \,  v^i,_k \,  \bar \p^4 A^k_i\,   dx + \mathcal{R} \\
& \le C \| \zeta \bar \p^3 q(t)\|_{\frac{1}{2}, B^+} \| \zeta \bar \p^4 A (t)\|_{H^ {\frac{1}{2}} (B^+)'}  + \mathcal{R}  \\
& \le C\sup_{t \in [0,T]} E(t)  + \mathcal{R}   \,,
\end{align*}
where we have used the pressure estimate (\ref{p_est}) and Lemma \ref{H12} for the last inequality.

Summing the estimates for $ \mathcal{I} _1$, $\mathcal{I} _2$, $ \mathcal{I} _3$ and integrating (\ref{Iintegrals}) from $0$ to $T$, we
obtain the inequality,
$$
\sup_{t \in [0,T]} \left( | \zeta \bar \p \eta(t) \cdot N_l  |_{3,B^0}^2 + \| \zeta  \bar \p^4 v(t)\|_{0,B^+}^2\right) \le M_0 +  T\, P( \sup_{t \in [0,T]} E (t)) \,.
$$
According to  Proposition \ref{div_est},
$$
\sup_{t \in [0,T]} \| \operatorname{div} v(t)\|^2_3  \le M_0 +  T\, P( \sup_{t \in [0,T]} E (t)) \,,
$$
from which it follows that
$$
\sup_{t \in [0,T]} \| \zeta  \bar \p^3 \operatorname{div} v(t)\|^2_{0, B^+}  \le M_0 +  T\, P( \sup_{t \in [0,T]} E (t)) \,.
$$
Hence, the normal trace estimate (\ref{normaltrace}) shows that
$$
\sup_{t \in [0,T]} \left( | \zeta \bar\p^4 v(t)\cdot N_l |_{- {\frac{1}{2}} ,B^0}^2\right) \le M_0 +  T\, P( \sup_{t \in [0,T]} E (t)) \,,
$$
from which it follows that 
$$
\sup_{t \in [0,T]} \left( | \zeta \bar \p \eta(t) \cdot N_l  |_{3,B^0}^2
+ | \zeta  \bar\p v(t)\cdot N_l |_{2.5,B^0}^2\right) \le M_0 +  T\, P( \sup_{t \in [0,T]} E (t)) \,.
$$
\end{proof}

Combining Proposition \ref{energyest} with the curl  estimates in  Proposition \ref{curl_est} and the divergence estimates in Proposition \ref{div_est} for $\eta(t)$ and $v(t)$ and using (\ref{hodge}) together with the fact that $v_t = - A ^T D q$ provides us with
 the following
\begin{theorem} \label{thm_apriori} Suppose that the initial pressure $p_0$ satisfies $\frac{\p p}{\p N} < 0$ on $\Gamma$ and
that $E(0) < \infty $.
 For  $T$ taken sufficiently small and for a polynomial function $P_2$,
\begin{align*}
\sup_{ t\in[0,T]} \left( \|  \eta(t) \|_{4.5,\Omega }^2 + \| v(t)\|_{4,\Omega }^2 + \| \operatorname{curl}_{\eta} v(t)\|_{3.5,\Omega }^2 + \| v_t(t)\|_{3.5,\Omega }^2 \right)   \le  M_0
+  T\, P_2({\sup_{t\in[0,T]}} E (t))   \,. 
\end{align*}
Moreover $\frac{\p p}{ \p n} < 0$ on $\Gamma(t)$ for $t\in [0,T]$.
\end{theorem}
%
(The rigorous construction of solutions to this problem was established in \cite{CoSh2007} using an approximation scheme founded on
the idea of horizontal convolution-by-layers.)   We next show that our solutions are continuous in time.

\begin{theorem}[Continuity in time]\label{appendix_thm2} The solution satisfies 
{\small
$$ \eta \in C([0,T]; H^{4.5}(\Omega))\,, \ v \in C([0,T]; H^{4}(\Omega)) \,, \  \operatorname{curl} _\eta v \in
C([0,T]; H^{3.5}(\Omega)) \,, \  v_t \in C([0,T]; H^{3.5}(\Omega)) \,.
$$}
\begin{proof} 
It follows immediately from Theorem \ref{thm_apriori} that
{\small
\begin{equation}\label{ccss1003}
 \eta \in C([0,T]; H^{4}(\Omega))\,, \ v \in C([0,T]; H^{3.5}(\Omega)) \,, \  \operatorname{curl} _\eta v \in
C([0,T]; H^{3}(\Omega)) \,, \  v_t \in C([0,T]; H^{3}(\Omega)) \,.
\end{equation}
 }
 
 Furthermore, by the same argument used to establish  (\ref{new4-13.2}) and (\ref{new4-13.3}), it follows that
 \begin{align} 
  \eta \in C([0,T]; H^{4.5}(\Omega)\operatorname{-w})\,, \ v \in C([0,T]; H^{4}(\Omega)\operatorname{-w}) \,, \nonumber \\
   \operatorname{curl} _\eta v \in C([0,T]; H^{3.5}(\Omega)\operatorname{-w}) \,, \  v_t \in C([0,T]; H^{3.5}(\Omega)\operatorname{-w}) \,,
   \label{ccss1004}
\end{align} 
the notation $H^s( \Omega)$-w denoting the weak topology.   Thus, it suffices to prove continuity of the norms
$$
 \|\eta(t)\|_{4.5}  \,,  \|v(t)\|_{4} \,,\|v_t(t)\|_{3.5}\,, \text{ and } \| \operatorname{curl} _\eta v(t)\|_{3.5} \,.
$$

For $h >0$ we define the horizontal difference quotient 
$$\bar \p^h  u:= {\frac{1}{h}} \left( u(\cdot+h{\bf e_i}) - u(\cdot) \right)\,, (i=1,2)\,,$$
and we proceed as in (\ref{ccss7001}), using $\bar\p^h \bar \p^3$ in place of $\bar \p^4$. The same energy estimate then yields
 $$
\frac{d}{dt} \left(\| \zeta  \bar \p^h \bar \p^3 v_l(t)\|_{0,B^+}^2 +  | \zeta  \sqrt{-q,_N}\bar \p^h \p^3 \eta_l(t) \cdot n_l( \eta_l)  |_{0,B^0}^2 \right) \le C \left( \| \eta(t) \|_{4.5}^2 + \| v(t)\|_4^2 \right) \,.
$$

With
$ \mathcal{F} _h (t) := \| \zeta \bar \p^h  \bar \p^3 v_l(t)\|_{0,B^+}^2$ and 
$ \mathcal{G} _ h (t) := | \zeta  \sqrt{-q,_N} \bar \p^h\bar \p^3  \eta_l(t) \cdot n_l( \eta_l)  |_{0,B^0}^2 $, we have that
$$
\frac{d}{dt} \left[
\mathcal{F}_h  (t)  +\mathcal{G} _h (t)
\right] \le C \left( \| \eta(t) \|_{4.5}^2 + \| v(t)\|_4^2 \right) \,.
$$
Integrating  from $t $ to $t+ \delta $, $0 <  \delta  \ll 1$,   and setting  $ \mathcal{H}_ h := \mathcal{F}_ h  + \mathcal{G}_h  $, we see that
\begin{equation}\nonumber
|\mathcal{H}_ h  (t+ \delta ) - \mathcal{H}_h (t) | \le \delta  C M_0 \,.
\end{equation} 
Since the bounds are independent of $ h >0$, we see that
\begin{equation}\label{ccss1020}
|\mathcal{H}  (t+ \delta ) - \mathcal{H}  (t) | \le  \delta C M_0 \,.
\end{equation} 
where $ \mathcal{H} = \mathcal{F} + \mathcal{G} $, and 
$ \mathcal{F}  (t) := \| \zeta   \bar \p^4 v(t)\|_{0,B^+}^2$ and 
$ \mathcal{G}  (t) := | \zeta  \sqrt{-q,_N}\bar \p^4  \eta(t) \cdot n( \eta)l  |_{0,B^0}^2 $

Hence, $t \mapsto \mathcal{H}(t)$ is uniformly Lipschitz continuous for $t \in [0,T]$.   Consider the product topology on the
Hilbert space
$ \mathcal{X} := L^2(B^+) \times L^2(B^0) $, with norm $ \| (f,g) \|_ { \mathcal{X} }^2 = \|f\|_{ L^2(B^+) }^2 + \|g\|_{ L^2(B^0) }^2$.
The convergence in the norm given by (\ref{ccss1020}) together with the continuity into the weak topology, given by
   (\ref{ccss1004}), show that $(\zeta \bar \p^4 v ,   \zeta  \sqrt {-q_{,N}} \bar \p^4 \eta \cdot n(\eta))$ are continuous into $ \mathcal{X} $. 
 We sum over all boundary charts;
thanks to (\ref{ccss1003}) and the elliptic estimate (\ref{ccss1001}), 
$q \in C([0,T]; H^{4}(\Omega))$, from which it follows that
$$
\bar \p^4 v \in C^0([0,T]; L^2(\Omega) ) \text{ and } \bar \p^2 \eta \cdot n(\eta) \in C^0( [0,T]; H^2(\Gamma)) \,.
$$

In order to prove that $\|\eta(t)\|_{4.5}$ is continuous for each $t \in [0,T]$, we will rely on the Lagrangian divergence and curl identities which
we established earlier.
  From equations  (\ref{ccss1000}) and  (\ref{ccss1011}), we see that
 \begin{align*}
\|\operatorname{curl} _{ \eta } D \eta(t+h)  -  \operatorname{curl} _{ \eta } D \eta(t)\|_{2.5} \le C h M_0 \,,
\end{align*}
so that $ \operatorname{curl} _\eta D\eta(t) \in C^0([0,T]; H^{2.5}(\Omega))$.   Similarly, from (\ref{ccss1012}), 
 \begin{align*}
\|\operatorname{div} _{ \eta } D \eta(t+h)  - \operatorname{div} _{ \eta } D \eta(t)\|_{2.5} \le C h M_0 \,,
\end{align*}
so that $ \operatorname{div} _\eta D\eta(t) \in C^0([0,T]; H^{2.5}(\Omega))$.   

It follows that for each $l=1,...,K$, 
\begin{alignat*}{2}
\operatorname{curl} _{ \eta_l} (\zeta_l \bar \p^2 \eta_l) & \in   && \  C^0([0,T]; H^{1.5}(B^+))\,, \\
\operatorname{div} _{\eta_l} (\zeta_l \bar \p^2 \eta_l) & \in && \  C^0([0,T]; H^{1.5}(B^+)) \,, \\
 \zeta_l \bar \p^2 \eta_l \cdot n_l( \eta_l) &  \in && \  C^0( [0,T]; H^2(B^0))\,.
\end{alignat*}
We let $w_l=   \bar \p^2 \eta \circ \eta ^{-1} $ denote the Eulerian counterpart to $\bar \p^2 \eta$, so that $w_l ( \cdot , t) : \eta_l(B^+, t) \to \mathbb{R}  ^3$.
Then, by the chain-rule, we see that, due to the continuity provided by (\ref{ccss1003}),
\begin{alignat*}{2}
\operatorname{curl}  w_l  & \in   && \  C^0([0,T]; H^{1.5}( \eta_l(B^+, t))\,, \\
\operatorname{div} w_l  & \in && \  C^0([0,T]; H^{1.5}( \eta_l(B^+, t)) \,, \\
w_l \cdot n_l &  \in && \  C^0( [0,T]; H^2( \eta_l(B^0, t))\,.
\end{alignat*}
We may then infer from Proposition \ref{prop1}, that 
$$
w_l \in C^0([0,T]; H^{2.5}( \eta_l(B^+, t)) \,,
$$
with bound depending only on $\eta_l \in C^0([0,T]; H^{2.5}(B^+))$.   It follows that for each $l=1,...,K$,  $\bar \p^2 \eta_l  \in C^0([0,T]; H^{2.5}(B^+))$.
It follows that $\bar \p^2 D\eta_l  \in C^0([0,T]; H^{1.5}(B^+))$, and hence the trace satisfies
 $D \eta_l  \in C^0([0,T]; H^{3}(B^0))$.   Summing over $l=1,...,K$, we see that
$$
 D\eta  \in C^0([0,T]; H^{3}(\Gamma)) \,.
$$

Therefore, we have the following elliptic system:
\begin{alignat*}{2}
\operatorname{curl} _{ \eta} (D \eta) & \in   && \  C^0([0,T]; H^{2.5}(\Omega ))\,, \\
\operatorname{div} _{\eta} ( D \eta) & \in && \  C^0([0,T]; H^{2.5}\Omega )) \,, \\
D\eta &  \in && \  C^0( [0,T]; H^3( \Gamma ))\,.
\end{alignat*}
Setting $W = D\eta \circ \eta ^{-1} $,  and using the fact that $\eta \in C^0([0,T]; H^4(\Omega))$ we see that
\begin{alignat*}{2}
\operatorname{curl}  W  & \in   && \  C^0([0,T]; H^{2.5}(\Omega (t))\,, \\
\operatorname{div} W  & \in && \  C^0([0,T]; H^{2.5}( \Omega (t)) \,, \\
W&  \in && \  C^0( [0,T]; H^3( \Gamma(t))\,.
\end{alignat*}
Elliptic estimates then show that
$$
W   \in  C^0([0,T]; H^{3.5}(\Omega (t))
$$
with a bound that depends on $\eta \in C^0([0,T]; H^4(\Omega))$ (but not on $\|\eta(t)\|_{4.5}$).  In turn, 
$D\eta \in C^0([0,T]; H^{3.5}(\Omega))$, and hence
$$\eta \in C^0([0,T]; H^{4.5}(\Omega)) \,.$$

Analogously, we find that $v \in C^0([0,T]; H^{4}(\Omega))$, which by elliptic estimates shows that
 $q \in C^0([0,T]; H^{4.5}(\Omega))$.  The momentum equation then shows that  $v_t \in C^0([0,T]; H^{3.5}(\Omega))$.
\end{proof} 
\end{theorem}

\section*{Acknowledgments}
We thank the referee for carefully reading the paper and for providing a number of  suggestions that improved the presentation.
DC was supported by the Centre for Analysis and Nonlinear PDEs funded by the UK EPSRC grant EP/E03635X and the Scottish Funding Council.  SS was supported by the National Science Foundation under grant
DMS-1001850, and by the United States Department of Energy through  Idaho National
Laboratory LDRD Project NE-156.

\end{document}